\documentclass[a4paper]{article}
\usepackage{a4, ams, amsmath, amssymb, graphics, subfigure, fullpage, color}%
\usepackage[pdftex]{graphicx}
\usepackage{fancyhdr}
\usepackage{multirow}
\usepackage{amsthm}
\usepackage[colorlinks]{hyperref}
\hypersetup{colorlinks,citecolor=green,filecolor=black,linkcolor=blue,urlcolor=blue}

\newtheorem{theorem}{Theorem}
\newtheorem{definition}{Definition}
\newtheorem{corollary}{Corollary}

\newtheorem{lemma}{Lemma}
\newtheorem{example}{Example}
\newtheorem{remark}{Remark}
\usepackage{upgreek} 

\newcommand{\MyBib}{./MyBib}

\usepackage[round]{natbib}

\usepackage{lscape}
\newcommand{\InclTDC}

\numberwithin{equation}{section}
\numberwithin{theorem}{section}
\numberwithin{corollary}{section}
\numberwithin{definition}{section}
\numberwithin{lemma}{section}
\numberwithin{remark}{section}
\numberwithin{example}{section}

\DeclareMathOperator{\E}{\mathbb{E}}
\renewcommand{\Q}{\mathbb{Q}}

\DeclareMathOperator{\var}{\textrm{var}}
\DeclareMathOperator{\filF}{\mathcal{F}}

\DeclareMathOperator{\e}{\textrm{e}}
\DeclareMathOperator{\ind}{1{\hskip -2.5 pt}\textrm{I}}
\DeclareMathOperator{\sign}{\textrm{sign}}

\DeclareMathOperator{\im}{\textrm{im}}
\DeclareMathOperator{\dom}{\textrm{dom}}
\newcommand{\beq}{\begin{equation}}
\newcommand{\eeq}{\end{equation}}
\newcommand{\beqn}{\begin{eqnarray}}
\newcommand{\eeqn}{\end{eqnarray}}
\newcommand{\bfig}{\begin{figure}}
\newcommand{\efig}{\end{figure}}
\newcommand{\btab}{\begin{table}}
\newcommand{\etab}{\end{table}}

\ifdefined \up 
	\renewcommand{\up}{\textrm{up}}
\else 
	\DeclareMathOperator{\up}{\textrm{up}} 
\fi

\title{Conic Martingales from Stochastic Integrals}

\author{Fr\'ed\'eric VRINS\thanks{Contact author. E-mail: \href{mailto:frederic.vrins@uclouvain.be}{frederic.vrins@uclouvain.be} $^\dag$The research of Monique Jeanblanc is supported by the Chaire March\'es en Mutation (F\'ed\'eration Bancaire Fran\c caise).}~~and~~Monique JEANBLANC$^\dag$\\ \vspace{0.1cm}\\
$^*$ \textit{Universit\'e catholique de Louvain}\\
Louvain School of Management\\
Chauss\'ee de Binche 151, 7000 Mons, Belgium\\
\&\\
Center for Operational Research and Econometrics (CORE)\\
Voie du Roman Pays 20, 1348 Louvain-la-Neuve, Belgium\\
\vspace{0.1cm}\\
\\
$^\dag$ \textit{Universit\'e d'\'Evry Val d'Essonne}\\
Laboratoire de Math\'ematiques et Mod\'elisation d'\'Evry (LaMME)\\
UMR CNRS 8071, \'Evry, France
}

\date\today

\begin{document}
\maketitle
\begin{abstract}
In this paper we introduce the concept of \textit{conic martingales}. This class refers to stochastic processes having the martingale property, but that evolve within given (possibly time-dependent) boundaries. We first review some results about the martingale property of solution to driftless stochastic differential equations. We then provide a simple way to construct and handle such processes. Specific attention is paid to martingales in $[0,1]$. One of these martingales proves to be analytically tractable. It is shown that up to shifting and rescaling constants, it is the only martingale (with the trivial constant, Brownian motion and Geometric Brownian motion) having a separable coefficient $\sigma(t,y)=g(t)h(y)$ and that can be obtained via a time-homogeneous mapping of \textit{Gaussian diffusions}. The approach is exemplified to the modeling of stochastic conditional survival probabilities in the univariate (both conditional and unconditional to survival) and bivariate cases.
\end{abstract}
\textbf{Keywords :}
Bounded martingale\textbf{}, stochastic differential equation, diffusion process, stochastic survival probability
\section{Introduction and motivation}
\label{sec:intro}
%

Mathematical finance extensively relies on martingales, mainly due
to the Fundamental Theorem of Asset Pricing. For instance, they
are used to represent the dynamics of (non-dividend paying) asset
prices, denominated in units of num\'eraire under an adequate
associated measure. As a consequence, under a true martingale
condition, asset price processes are   governed by conditional
risk-neutral expected values of future (discounted) cashflows.
Martingales are also central in measure change techniques (via
Radon-Nikodym derivative processes).

Depending on the situation, the martingale processes may be
subjected to some constraints. Discounted stock prices and
Radon-Nikodym derivative processes are positive. Therefore,
exponential martingales, which meet the non-negativity constraint
are very popular tools.

Financial processes can be subjected to other constraints, like
being bounded below \textit{and} above. This is for instance the
case of discounted zero-coupon bond prices in the case where
interest rates (short rate $r_s$) cannot be negative:
\beq 
P_{t}(T)=\E\left[ \exp - \int_t^T r_sds \Big|\filF_t\right],\nonumber
 \eeq
which belongs to $[0,1]$ almost surely, and thus so is the
martingale $P_{t}(T) \e^{-\int_{0}^t r_sds}$. Similarly,
conditional survival probabilities $S_{t}(T)$ (probability that a
default event $\tau$ occurs after a given time $T$ as seen from
time $t$), defined as the conditional expected value of survival
indicators
\beq 
S_{t}(T)=\E\left[\ind_{\{\tau>T\}}|\filF_t\right]\label{eq:DefStT} 
\eeq
are   $(\filF_t)_{t\geq 0}$-martingales valued in $[0,1]$. Note
that here, we have to deal with a  family of martingales
depending on the parameter $T$ and that $S_t(T)$ has to be,
for any $t$, decreasing with respect to $T$. Note that $\E[S_t(T)]=  \Q\{\tau >T\}=S_0(T)$.

Surprisingly however, bounded martingales received little
attention. In the case of survival
probabilities modeling for instance, practitioners often disregard
inconsistencies, working with Gaussian processes (which are not
constrained to evolve in the unit interval) instead (see
e.g.~\cite{Ces09}). This also applies to many standard approaches,
where shifted Ornstein-Uhlenbeck (Hull-White) or shifted
square-root diffusion (SSRD, also known as CIR$^{++}$) are used as
intensity processes, and may lead to probabilities exceeding
1.\footnote{The shift in the square-root process is required in
order to fit CDS quotes, and may indeed affect the positivity of
the resulting stochastic intensity.}

In spite of these drawbacks, these methods remain popular. We
believe this results from the lack of adequate and tractable
alternatives. Our purposes here is precisely to propose a
contribution to fill this gap. We introduce the concept of
\textit{conic martingales} (a naming that we justify in the
sequel), which corresponds to the intuitive idea of ``martingales
evolving between bounds''. We shall see how such processes can be
handled and simulated in such a way that the paths stay within the
bounds. We further study some properties like the implied
distribution and the asymptotic behavior. Specific interest is
dedicated to martingales obtained by mapping Gaussian processes
through functions with image being a compact set.

The paper is an extension of~\cite{Vrins14,Vrins15a}, and is
organized as follows. We first sketch the model setup in
Section~\ref{sec:Setup} and recall some results related to
existence, uniqueness and martingale property of driftless
stochastic differential equations (SDE) in
Section~\ref{sec:results}. The concepts of \textit{cone} and
\textit{conic martingales} are then introduced in
Section~\ref{sec:cones}. We then discuss how those can be
constructed in Section~\ref{sec:construction} and focus on one
particular process (Section~\ref{sec:Phi}). Finally, we apply this
process to the modeling of univariate (conditional and
unconditional) and bivariate survival probability modeling in
Section~\ref{sec:SP} before concluding.

\section{Setup}
\label{sec:Setup}
%

We consider a filtered probability space
$(\Omega,\filF,\mathbb{F}=(\filF_t)_{t\geq 0},\mathbb{Q})$ where
$\mathbb{F}$ is a Brownian filtration hence, any martingale
is continuous. In the sequel, all the processes are defined on
$(\Omega,\mathbb{F})$. We study the martingale property of the
solution $Y$ to a ``driftless'' SDE of the form
\beq
dY_t=\sigma(t,Y_t)dW_t~,~0\leq t\leq T\label{eq1}
\eeq
where $W$ is a $\mathbb{Q}$-Brownian motion adapted to the
filtration $\mathbb{F}$ and, for each $t$, the random variable $Y_t$
is restricted to be in some interval.

 In the sequel, we shall omit the specification
``$0\leq t\leq T$'' in case of non-ambiguity, and the
diffusion coefficient function $\sigma(t,y):[0,T]\times
\mathbb{R}\to A\subseteq\mathbb{R}$ is always assumed to be
continuous in $y$ and \textit{Borel-measurable} in $t$.

The process $Y $ is a \textit{local martingale}, but may fail to
be a martingale. Additional technical conditions are required to
prevent the use of ``too fancy'' diffusion coefficient  functions.
However, $\sigma(.,.)$ does not need to be ``that fancy'' for the
(global) martingale property to be lost. For instance, the
solution to~(\ref{eq1}) is a   martingale if $\sigma(t,x)=x$
(Geometric Brownian motion) but is a strict local martingale when
$\sigma(t,x)=x^2$. The distinction between local and global
martingale is   crucial in financial applications in order to
prevent arbitrage opportunities and bubbles \cite{CoxH05},
\cite{Prott13}. To determine whether the solution to the above SDE
is indeed a martingale, the following square integrability
condition can be useful (see e.g.,~\cite{Kara05},~\cite{RevYor99}
and Section 4.9 of~\cite{Shrev04})

\begin{theorem}\label{th:martingale}
The stochastic  process  $\int_0^{\cdot}\sigma(s,Y_s)dW_s$ is a martingale on $[0,T]$ if
\beq
\E\left[\int_{0}^T \sigma^2(s,Y_s)ds\right]<\infty\label{eq1m}
\eeq
\end{theorem}
%

Condition~(\ref{eq1m}) may be difficult to check, since it
requires to have some information about the solution to the SDE
(which may even not exist). In this context, not a lot can be said
at this stage about the martingale property of $Y$. However, some useful results can be found, based on the shape of the diffusion coefficient $\sigma(t,y)$ only, as reviewed in the next section.

\section{General results on martingale property of It\^o stochastic integrals}
\label{sec:results}
%
In this section, the martingale property of $Y$ is discussed form the properties of
the deterministic diffusion coefficient function $\sigma $. The main result in that respect is the following (Theorem 2.9
in~\cite{Kara05}):
\begin{theorem}\label{th:LipTimeDep}
Let $\sigma(t,y)$ be Lipschitz in $y$ for all $t\geq 0$. 
In addition, suppose $\sigma$ satisfies the sub-linearity condition
\beq |\sigma(t,y)|\leq C(1+|y|)\label{eq3} \eeq
for some constant $C<\infty$, then eq.~(\ref{eq1}) has a pathwise unique (and thus strong)
 solution $Y $ satisfying $\E\left[\int_{0}^T Y_s^2 ds\right]<\infty$. Moreover, $Y$   is a martingale.
\end{theorem}
Condition~(\ref{eq3}) aims at preventing explosion, while the Lipschitz constraint typically guarantees
existence and uniqueness. In particular, the solution to~(\ref{eq1}) does not explode.

This result ensures that the solution to $dY_t=g(t) Y_t dW_t$ is a
martingale for bounded functions $g$ on $[0,T]$. However, it is
not enough to ensure that the solution to the square-root
driftless SDE $dY_t=\sigma\sqrt{Y_t}dW_t$ is also a martingale
($\sqrt{x}$ fails to be  Lipschitz in any interval containing $
0$). The existence is formally proven in~\cite{Zvon74} by
replacing the Lipschitz continuity by a Holder-$1/2$ one.

The class of admissible diffusion coefficients in this theorem can be
significantly extended by replacing the global Lipschitz condition
by a local one. This covers quite a large class of coefficients
since every continuously differentiable function is locally
Lipschitz (see~\cite{Kloed99}). Yet another extension relies on
the \textit{Yamada-Watanabe} condition only~(see
e.g.~\citep{Kloed99} Section 4.5 p.134-135 and~\citep{Kara05}
\cite[Section 5.5.5]{jyc:3m}). 

It is worth noting that a sufficient and necessary condition exists in the time-homogeneous case $\sigma(t,y)=\sigma(y)$ when $\sigma(y)=0$ for all $y\leq 0$. The positive process $\int_0^{\cdot}\sigma(Y_s)dW_s$ is a martingale if and only if $x/\sigma^2(x)$ is \textit{not} integrable near infinity (see e.g.~\cite{Carr07},~\cite{Delb02}).
These generalizations are however
not necessary in the context of this paper.

\section{Cones of stochastic processes and conic martingales}
\label{sec:cones}

The above framework depicts the general context associated to the
martingale property of solutions to driftless SDEs. We would like
now to focus on bounded processes, and introduce the concept of
\textit{cone}.
\begin{definition}[Bounded process]
A stochastic process $ {Y} $ is \emph{bounded} on $[0,T]$ if
there exists a constant $M $ such that $\sup_{t\in[0,T]} |Y_t|<M$.
It is \emph{locally bounded} on $\mathbb{R}$ if for all $t>0$ there exists a
constant   $M(t) $ such that $ sup_{s\leq t} |Y_s|<M(t)$.
\end{definition}

We recall that the range $R(X)$ of a random variable $X$ is the
support (which is a closed set, see~\cite[Ch. 3-50]{Dell75}) of its distribution. By extension, we define the
\textit{range envelope} of a stochastic process.
\begin{definition}[Range envelope]
The \emph{range envelope}
 $\mathcal{S}:= \big(R(Y_t)\big)_{t>0}$ of a stochastic process
 $ Y$ is the time-indexed sequence of the ranges $R(Y_t)$ of $(Y_t)_{t>0}$.
\end{definition}
Observe that by continuity, the range envelope of a continuous process is a sequence of connected sets.
\begin{definition}[Cone]
If for each $t>0$ there exists $M(t)>0$ such that $R(Y_t)
\subset [-M(t),M(t)]$ and the sequence $R(Y_t)$ is increasing  in
the sense that for all $0< s\leq t$ we have $R(Y_s)\subseteq
R(Y_t)$, we say that $\mathcal{S}$ is the \emph{cone associated
to the stochastic process $Y$} or simply, the \textit{cone of
$Y$}.\footnote{Observe that in the definition, the cone of $Y$ is
defined for $t>0$; this is because the range of $Y_t$ is trivially
equal to the constant $Y_0$ at $t=0$.}
\end{definition}

Note that the concept of \textit{range envelope} differs from the
notion of \textit{envelope} in two points~\cite{Vene79}. First,
the latter is a set-valued \textit{random} process which
upperbounds $|Y_t|$ with probability one during any finite length
of time, while the former is a \textit{deterministic} time-indexed
sequence of compact sets.

Second, the range envelope (if it exists) is unique
since it corresponds at any time $t$ to the \textit{smallest} set
to which $Y_t$ belongs with probability 1.
\begin{definition}[Conic process]
A stochastic process is said to be \emph{conic} if its range
envelope is a cone. When the stochastic process is a martingale,
we say that it is a \emph{conic martingale}.\label{def:conproc}
\end{definition}
In this context, the word ``conic'' refers to the fact that the
range of the process is bounded and is non-decreasing with time,
so that the upper (resp. lower) bound of $Y_t$ increases (resp.
decreases) with $t$.

Notice that only processes that are bounded up to $T$ admit a cone.
For instance, neither Brownian motion
nor its Dol\'eans-Dade exponential are conic processes in the
sense of the above definitions. 
In the one dimensional case, any martingale
evolving between two deterministic real-valued functions of time
$a(t),b(t)$ satisfying $-M<a(t)\leq b(t)<M$ for all
$t\in\mathbb{R}_+$ admits a cone.

The next corollary
shows that \textit{conic} and \textit{locally bounded} martingales
are in fact a same thing.
\begin{corollary} \label{conicmart}
Any \textit{conic martingale} is a locally bounded martingale.
Reciprocally, any locally bounded continuous martingale is
a \textit{conic martingale}.
\end{corollary}
Note that on $[0,T]$, conic martingales $X$ are martingales such that $X_T$ is bounded.

\begin{proof}
The first assertion is obvious: indeed, by definition, a conic
martingale $Y $ is a martingale evolving between bounds
$-\infty<a(t)\leq Y_t\leq b(t)<\infty$; it is therefore locally
bounded.  Let us prove the converse. Let $Y $ be a martingale such
that for all $t>0$ there exists $a(t),b(t)\in\mathbb{R}:a(t)\leq
Y_t\leq b(t)$. Because the paths of $Y $ are almost surely
continuous, the support of $Y_t$ cannot have ``holes'' so that
$R(Y_t)=[a(t),b(t)]$ are closed intervals. 
It remains to prove that the time-indexed sequence of ranges is
increasing. To that end, suppose there exists $y_s\in R(Y_s)$ that
does not belong to $R(Y_t)$ for some $t>s$. Then, either $y_s>\sup
R(Y_t)$ or $y_s< \inf R(Y_t)$. In both cases,
$\E[Y_t|\filF_s]\neq y_s$ as $Y_t>y_s$ or $Y_t<y_s$ with
probability one, respectively; $Y $ fails to be a martingale.
Consequently, a necessary condition for $Y $ to be a martingale is
that the range of $ {Y}$ is an increasing sequence.
\end{proof}
In order for the solution $Y$ of a driftless SDE to exist, we must
have $\int_0^t\sigma^2(s,Y_s)ds<\infty$ almost surely. It is then
a local martingale, see e.g.~\citet{Kara05}. The following
well-known result (Th. 5.1 in  \cite{Prott05})  fills the gap
between local and genuine martingales in the special case of
(globally) bounded processes.
\begin{theorem}\label{th:locmboundm}
Every bounded local martingale is a martingale.
\end{theorem}
In the sequel, we shall focus on separable
diffusion coefficients, defined below.
\begin{definition}[Separable diffusion coefficient]
A diffusion coefficient $\sigma(t,x)$ is said \textit{separable}
if  it can be written as \beq
\sigma(t,x)=g(t)h(x)\label{eq:SDEySep} \eeq \label{def:sepcoef}
with $h$ continuous.
\end{definition}
It is obvious that when the diffusion coefficient is separable
where the time component $g $ is a bounded function for all $t\geq
0$, the solution $Y $ to SDE~(\ref{eq1}) with initial condition
$Y_0\in [a,b]$ is a continuous martingale in $[a,b]$ provided that
function $h(x)$ is continuous on $[a,b]$ and satisfies $h(x)>0$
for $x\in (a,b)$ and $h(x)=0$ for $x\in\{a,b\}$.

Obviously, analysis of martingales with constant cone can
be restricted to martingales with standard cone $[0,1]$ as any
martingale in $[a,b]$ can be obtained from a martingale in
$[0,1]$.
We thus have the following corollary. 
\begin{corollary}[]\label{cor:sepmart}
Consider the separable diffusion coefficient in~eq.(\ref{eq:SDEySep}) where $g(t)>0$, $h(x)>0$ for $x\in]a,b[$ and $h(x)=0$ elsewhere. Then, if SDE~(\ref{eq1}) admits a unique solution $Y $ satisfying $Y_0\in]a,b[$, $Y\in [a,b]$. It is therefore a bounded local martingale, and from Theorem~\ref{th:locmboundm}, a genuine martingale.
\end{corollary}
\begin{proof} The SDE $dY_t= g(t) h(Y_t) \ind_{\{Y_t \in [a,b]\}}
dW_t$ admits a unique solution, which is a Markov process.
Denote by $\tau$ the first hitting time of the boundary. In the
case $\tau <\infty$, one obvious solution $Y$ on $[\tau,\infty)$ given $Y_\tau$ is
$Y_t= Y_\tau\in\{a,b\}$. Therefore, the solution $Y$ to the above SDE is unique and belongs to $[a,b]$.
Because $Y\in[a,b]$ $\Q$-a.s. and $h(a)=h(b)=0$, the solution to $dY_t= g(t) h(Y_t) \ind_{\{Y_t \in [a,b]\}}
dW_t$ is the same as that of $dY_t= g(t) h(Y_t)dW_t$.
\end{proof}

\section{Construction of conic martingales}
\label{sec:construction}
%
Obviously, any martingale   with cone in $[0,1]$  is of the form
$\E[\zeta \vert \filF_t]$ for a random variable $\zeta$, valued in
$[0,1]$. However, it is not possible to  compute the
diffusion coefficient of such martingales, since these martingales
are not always diffusion processes.

The above section shows how martingales evolving in the
compact set $[a,b]$ can be obtained, by adequately choosing the
diffusion term $\sigma(t,x)$ in eq.~(\ref{eq1}). This framework is
interesting theoretically but may be hard to deal with in
practice. To illustrate this, suppose we want to construct a
martingale in $[0,1]$. To this end, we can choose
$Y_0\in[0,1]$, $\sigma(t,x)=g(t)h(x)> 0$ for all $t>0$ and
$x\in(0,1)$ where $h =0 $ for $x\notin(0,1)$. One simple
``smooth function'' satisfying the required conditions is
\beq \sigma(t,x)=\eta x(1-x) \nonumber
\eeq
where $\eta$ is a constant, which is proven to satisfy the
conditions of Theorem~\ref{th:LipTimeDep}, and $Y\in[0,1]$ if
$Y_0\in [0,1]$.\footnote{In the sequel, we shall deal with processes that belong to $[0,1]$ almost surely. In that case, we shall restrict ourselves to specify the diffusion coefficient $\sigma(t,x)$ on $\mathbb{R}^+\times [0,1]$. Of course, the later can trivially be extended to $x\in\mathbb{R}$ via the indicator function $\ind_{\{x\in(0,1)\}}$. This preserves the dynamics of the process and allows us to rely on existence and uniqueness results, which require the SDE coefficients to be defined for $x\in\mathbb{R}$; see e.g. proof of Corollary~\ref{cor:sepmart}} This leads to the SDE:
\beq
dY_t=\eta Y_t(1-Y_t) dW_t,~~Y_0\in[0,1]\label{eq:y1my}
\eeq
Theorem~\ref{th:LipTimeDep} ensures that this SDE admits a
solution which, from Corollary~\ref{cor:sepmart}, is a
martingale. Moreover, because $\sigma(t,0)=\sigma(t,1)=0$ and
$\sigma(t,x)>0$ for $x\in(0,1)$, the range of this process is
$[0,1]$ provided that $Y_0\in[0,1]$; it is thus a conic martingale
with cone $[0,1]$, as per Definition~\ref{def:conproc}. However,
even if numerical schemes can be worked out to estimate the
distribution of $Y_t$, the analytical expression of such atypical
SDE may not be trivial to find, if existing. Moreover, from a
practical perspective, such schemes need to guarantee that all
paths (for Monte-Carlo simulation, or the range of the
distribution, for PDE solver) of $Y$ remains in the cone
associated to the theoretical solution (that one may guess from
the SDE). Generally speaking,  implicit schemes satisfying the
boundedness conditions are required, but can be tedious to find
out.

 To address these two issues, we propose a specific
construction scheme that yields the conic martingale $Y $ as a
transformation of a simpler (unconstrained, or ``free'') process
$X $ via some smooth functional $Y_t=F(t,X_t)$. The SDE and
statistics of $Y $ can thus directly be obtained through those of
$X $.

Here below we first show how one can create one-dimensional martingale with cone $\mathcal{S}=[0,1]$.

\subsection{Methodology}
\label{subsec:construction:cstcone}
%
We now proceed with the next result, which is an important tool
for constructing conic martingales.

We shall need  conditions ensuring existence and uniqueness of
solutions to generic SDE of the form
\beq
dX_t = \mu(t,X_t)dt+\eta (t,X_t) dW_t,\label{eq:LatProc00}
\eeq
which  can be found (see, e.g., \cite{Osk03},\cite{Kloed99},
\cite[Sections 1.5.4 and 1.5.1]{jyc:3m}, \cite[Chapter IV, Section
3]{RevYor99}).
\begin{theorem}[Autonomous Mapped Martingales]\label{th:AMM}
Let $F(x):dom(F)\to[0,1],x\mapsto F(x)$ be a strictly monotonic
function of class $\mathcal{C}^2$ with bounded first
derivative. Note $f(x):=F'(x)$ and $f'(x)=F''(x)$. Let $\eta$
be a function defined on  $\mathbb{R}^+\times dom(F)$. Assume that there
exists a process $X $ with $R(X_t)\in \dom(F)$ solution of the
stochastic differential equation (SDE)
\beq
dX_t=\frac{\eta^2(t,X_t)}{2}\psi(X_t) dt + \eta(t,X_t)dW_t\label{eq:LatProc0}
\eeq
where $\psi(x):= -\frac{f'(x)}{f(x)}$ is the \textit{score function} associated to $F$. Then, the process
\beq Y_t=F(X_t)~,~~t\geq 0 \nonumber\eeq
is a martingale in $[0,1]$. If the range of $X$ coincides with $dom(F)$, the range of $Y$ is $[0,1]$. In this case, $Y $ is called a \textit{conic martingale with cone $[0,1]$} or equivalently,  a
\textit{$[0,1]$-\textit{martingale}}.
\end{theorem}
\begin{proof} The process $Y =F(Z)$, valued in$[0,1]$,  has dynamics  given via It\^o's lemma:
\beq
dY_t= f(Z_t)\eta(t,Z_t)dW_t\nonumber
\eeq

The process $Y$ is then a bounded local martingale hence a martingale.
\end{proof}
Notice that since $F$ is a bijection, it is invertible, and the
dynamics for $Y $ becomes
\beq dY_t=f\circ
F^{-1}(Y_t)\eta(t,F^{-1}(Y_t))dW_t=:\sigma(t,Y_t)dW_t\label{eq:SDEYonly}
\eeq
so that $Y$ is a diffusion.

\begin{corollary}
Let $F$ be a bijection of class $\mathcal{C}^2$ and $X$ be the solution to eq.~(\ref{eq:LatProc0}). Then $Y=F(X)$ is a
conic martingale with cone $[0,1]$ satisfying the
SDE~(\ref{eq:SDEYonly}). Moreover, the cumulative distribution
function of $Y_t$ ($F_{Y_t}(y)$) can be obtained form that of the latent random variable $X_t$. In
particular, if $F$ is increasing:
\beq
F_{Y_t}(y)=F_{X_t}(F^{-1}(y))\label{eq:CDFtransform}
\eeq
\end{corollary}
\begin{remark}
Simple candidates for function $F$ are cumulative distribution (or
survival distribution) functions defined on the real line,
admitting a continuously differentiable density and invertible.
\end{remark}
%
\ifdefined \inclTDC
    \subsection{Examples (constant cone)}
\else
    \subsection{Examples}
\fi
\label{subsec:construction:ex}
%
\begin{enumerate}
\item Let $F(x)= e^{-\lambda x}$ and $\eta(t,x)=\eta x$. Then
 $X$ satisfies a variant of the \textit{Verhulst} equation:
\beq dX_t=\lambda (\eta^2/2)   X^2_tdt+\eta X_t dW_t\label{eq:ExpSDE}
\eeq
with solution  
\beq X_t=\frac{\Theta_t}{1- (\lambda\eta^2/2)
\int_{0}^t\Theta_sds}~,~~\Theta_t:=
X_0\e^{-(\eta^2/2) t +\eta W_t} 
\eeq
up to explosion time $\tau:=\inf\{t\,:\, \int_{0}^t\Theta_sds =2/ \lambda\eta^2\}$.

The process $Y$ defined as $Y_t=\exp(-\lambda X_t), t\leq  \tau$ defines  a martingale (valued in $[0,1]$) with  SDE  given by
\beq dY_t=-\eta\ln(Y_t)Y_tdW_t\label{eq:lnY} \eeq Note that, by
construction, $\inf \{t\,:\,Y_t\in\{0,1\}\}=\inf
\{t\,:\,Y_t=0 \}=:\tau$. The boundary $1$ is not reached (up to $\tau$) and
$Y_{\tau\vee t}=0$.

 \item Let us come back to the
SDE~(\ref{eq:y1my}). We can see that it consists in
eq.~(\ref{eq:SDEYonly}) with $\sigma(t,x)=\eta x(1-x)$. Setting
$\eta(t,x)=\eta$ in~(\ref{eq:LatProc0}), it appears that we must
have $f(F^{-1}(y))=y(1-y)$. Changing the variable
$x=F^{-1}(y)$ leads to the (logistic) first-order non-linear
differential equation
\beq f(x)=\frac{d F(x)}{dx}= F(x)(1-F(x))\nonumber \eeq
which solution is proven to be
\beq F(x)=\frac{c\e^{x}}{1+c\e^{ x}}~,~~\psi(x)=2F(x)-1\nonumber
\eeq
This function is the cumulative distribution function of a logistic random variable with mean $-\ln(c)$ and variance $\pi^2/3$.\footnote{Note that since $c\e^{X_0}=F(X_0)/(1-F(X_0))$ and $\dom(F)=[0,1]$, we must have $c>0$.}
In other words, it appears that the conic martingale process $Y$ defined by eq.~(\ref{eq:y1my}) can be obtained by simply mapping through the above (distribution) function $F$ the unconstrained (latent) process $X$ (defined by the SDE~(\ref{eq:LatProc0})) which instantaneous variance is set constant $\eta^2(t,x)=\eta^2$ and the drift, as per Theorem~\ref{th:AMM}, given by $\frac{\eta^2}{2}\psi(x)$ where $\psi(x)=-\frac{f'(x)}{f(x)}=2F(x)-1=\tanh(x/2)$ is the score function associated to the above distribution function. In the specific case where $c=1$, the SDE of the latent process writes
\beq dX_t=\frac{\eta^2}{2} \tanh\left(\frac{X_t}{2}\right)dt+\eta
dW_t\label{eq:tanhX} \eeq 
The coefficients of this SDE satisfy the usual conditions ensuring existence and uniqueness of a solution $X$, so that the solution $Y$ to SDE~(\ref{eq:lnY}) exists and is unique, too.
\item It is obvious that a same SDE for $Y=F(X)$
can be obtained from various combinations of $(X,F)$. In the
above example, $c$ can be chosen to be any positive scalar, but
provided that the correct drift (score function) is used for
$X$, the same SDE is obtained for $Y$. Similarly, a given
latent process $X$ can lead to several driftless equations for
$Y=F(X)$, depending on the choice of $F$. This results from
the fact that different mappings $F$ can lead to the same score
function. Therefore, for a given drifted process $X$, one can
find several mappings $F$ such that the resulting SDE of
$Y=F(X)$ is driftless. For instance, setting $F(x)=\tanh(x/2)$
leads to $\sigma(t,y)=\frac{\eta}{2}(1-y^2)$:
\beq
dY_t=\frac{\eta}{2}(1-Y_t^2)dW_t\label{eq:tanhSDE}
\eeq
The score function of $F(x)=\tanh(x/2)$ is equal to $\tanh(x/2)$
as well, which is the same as that of $F(x)=\e^x/(1+\e^x)$ that
led to SDE~(\ref{eq:y1my}).
\item Consider the standard Gaussian
case, $F(x)=\Phi(x)$. The score function of the Gaussian
distribution is given by $\psi(x)=-\phi'(x)/\phi(x)=x$ where
$\phi(x)$ is the standard Normal density function, and we get that
$Y=\Phi(X)$ is a $[0,1]$-martingale provided that the SDE
\beq dX_t=\frac{\eta^2(t,X_t)}{2} X_t dt+\eta(t,X_t)dW_t
\label{Xgaus}
\eeq
has a solution and $Y$ has diffusion coefficient
\beq
\sigma(t,y)=\phi\circ\Phi^{-1}(y)\eta\left(t,\Phi^{-1}(y)\right)\label{eq:PhiSDE}
\eeq
\end{enumerate}

\subsection{Another conic martingale}
\label{sec:tdc}

The standard Normal distribution function $\Phi$ can be used to
turn any It\^o integral into a martingale in $[0,1]$ with given
initial value $x\in(0,1)$.

Define
\beq
 Z_t:=z+\int_0^t\sigma_sdW_s\nonumber
\eeq
where $z\in\mathbb{R}$ and $\sigma $ is adapted to the natural
filtration of $W$.
\begin{lemma} Let $x\in(0,1)$ and set $z:=\Phi^{-1}(x)$.  For
$\Upsilon :=1/\sqrt{1-\left[ Z\right] }$, setting $\tau =
\inf\{t\,:\, [Z]_t=1\}$, the stochastic process $\mathcal{M}(Z)$ defined as
\beq
\mathcal{M}(Z )_t:=\Phi\left(\Upsilon_t Z _t\right), \, t<\tau\label{eq:DDPhi}
\eeq
is a martingale in $[0,1]$ with initial value $x$. Moreover, if $R(X_t)=\mathbb{R}$, then $\mathcal{M}(Z )_t$ is a $[0,1]$-martingale.
\end{lemma}
\begin{proof}
It is obvious to see that $\mathcal{M}(Z  )\in[0,1]$ and that
$\mathcal{M}(Z )_0=x$ since $\Upsilon_0=1$ and $Z _0=z$. From
Theorem~\ref{th:locmboundm}, it remains to show that it is a local
martingale. This is straightforward using the  property
$\Phi''(x)=-x\Phi'(x)=:-x\phi(x)$ and the fact that
$d\Upsilon_t=\frac{\Upsilon_t^3}{2}d\left[ Z\right]_t$. Indeed,
from It\^o's lemma,
\beqn
\frac{d\mathcal{M}(Z )_t}{\phi\circ\Phi^{-1}(\mathcal{M}(Z )_t)}&=&Z _td\Upsilon_t+\Upsilon_tdZ_t-
\frac{\Upsilon_tZ _t}{2}\Upsilon_t^2d\left[Z\right]_t\nonumber\\
&=& \Upsilon_tdZ_t \nonumber
\eeqn
\end{proof}

The mapping $\mathcal{M}$ turns any continuous local martingale $Z $ defined on $\mathbb{R}$ into a $[0,1]$-martingale $\mathcal{M}(Z )_t=\Phi(Z_t/\sqrt{1-[Z]_t})$ using function $\Phi$. This is similar to the Dol\'eans-Dade exponential $\mathcal{E}$, which maps $Z $ to a non-negative martingale $\mathcal{E}(Z )_t=\exp\{Z_t-[Z]_t/2\}$ using the exponential function provided that $Z$ satisfies the Novikov condition. This results from the connections between first and second derivatives of these functions. From this perspective, the martingale $\mathcal{M}(Z )$ can be seen as the equivalent of $\mathcal{E}(Z )$ but for $[0,1]$-martingales instead of $\mathbb{R}^{+}$-martingales.

\begin{remark}An important point is that the process $\mathcal{M}(Z )$
reaches the bounds if and only if $\left[ Z\right]$ can reach 1.
Assume for instance a constant diffusion coefficient
$\sigma_t=\eta\in\mathbb{R}_0^+$. Then, $[Z]_t=\eta^2 t$ so that
$\frac{ Z_t}{\sqrt{1-\left[ Z\right]_t}}$ is a.s. finite for
$t<\eta^{-2}$ but $\mathcal{M}(Z )_{\eta^{-2}}\in\{0,1\}$ a.s.,
i.e. we reach (and stick to) one of the boundaries at
$t=1/\eta^2$. On the other hand, $\mathcal{M}(Z  )$ cannot reach
the boundaries if we choose $\sigma_t=\sigma(t)=\eta\e^{-\eta^2
t/2}$ since $\int_0^\infty\sigma_s^2ds=1$. Moreover, the
process $\mathcal{M}(Z  )$ is a diffusion if $\sigma_t=\sigma(t)$
is a deterministic function of time. In that case, the diffusion
coefficient of $\mathcal{M}(Z  )$ is separable in the sense of
eq.~(\ref{eq:SDEySep}) with
$g(t)=\sigma(t)\Upsilon_t=\sigma(t)/\sqrt{1-\int_0^t \sigma^2(s)ds}$
and $h(x)=\phi\circ\Phi^{-1}(x)$.
\end{remark}
%
\subsection{Practical considerations}
%
As explained above, conic martingales can be obtained by
specifying the form of the diffusion coefficient $\sigma(t,y)$.
However, the resulting SDE's are most often not analytically
tractable, so that numerical schemes need to be used.

The above considerations suggest that it may be better to first
(analytically or numerically) try to solve the SDE of an
underlying free process $X=G(Y)$, and then get the solution of $Y$
via the mapping $F=G^{-1}$. If $G$ is chosen in a clever way, it
could be that the second SDE is more easy to deal with (more
standard, evolving in $\mathbb{R}$ instead of e.g. $[0,1]$). At
least we can get the correct range. This is the purpose of the
next theorem, which tells us how to choose $G=F^{-1}$ so that
$X_t=G(Y_t)$ takes a specific, more appealing form. More
specifically, the developed methodology allows us to write the
solution (if it exists) of the SDE~(\ref{eq1}) when the
instantaneous volatility is separable (in the sense of
eq.~(\ref{eq:SDEySep})) as a mapping $F$ of another process with
specific drift but diffusion coefficient $g(t)$ depending on time
only. The below theorem states sufficient conditions for the
solution $Y$ to the SDE~(\ref{eq1}) to be given by $F(X)$ where
$X$ is as desired. Moreover, (i) we are told which $F$ we have to
choose in order for $X_t=F^{-1}(Y_t)$ to have the required
dynamics, and (ii) the SDE of $X$ is completely specified.
\begin{theorem}
Consider the SDE~(\ref{eq1}) where
the diffusion coefficient is separable in the sense
of~(\ref{eq:SDEySep}). Assume function $g$ satisfies
$g^2(t)<\infty$ for all $t$ and $h(y)$ is strictly positive of
class $\mathcal{C}^2$ in the set
$\mathcal{A}\setminus\partial\mathcal{A}$ and vanishes at the
(existing) boundaries $\partial\mathcal{A}$ of $\mathcal{A}$. Let
$y=F(x)\subseteq\mathcal{A}$ solve the first order autonomous
non-linear ODE\footnote{The solution to this ODE is proven to have
the general form $y(x)=H(x+k)$ where $k$ is the integration
constant and $H(x)$ is the inverse of $\int_{\inf \dom(h)}^x
\frac{1}{h(u)}du$}
\beq
\frac{dy}{dx}=h(y)\label{eq:ODE1}
\eeq
If $\psi(x)=-F''(x)/F(x)$ is Lipschitz continuous, then~(\ref{eq1}) admits the strong pathwise unique solution $Y_t=F(X_t)$ where $X$ is the unique strong solution to the SDE~(\ref{eq:LatProc0}) with initial value $X_0:=F^{-1}(Y_0)$, diffusion coefficient $\eta(t,x)=g(t)$ and drift $\mu(t,x)=-\frac{g^2(t)F''(x)}{2F'(x)}=\frac{g^2(t)}{2}\psi(x)$.
\end{theorem}
\begin{proof}
Let us first prove that the solution $y=F(x)$ to the non-linear differential equation~(\ref{eq:ODE1}) is invertible, of class $\mathcal{C}^2$. Because $h(x)>0$ for all $x\in\mathcal{A}$ and $\im(F)\subseteq \dom(h)$, $F(x)=\int_{-\infty}^x h(F(u))du+k$ is continuously (strictly) increasing; $F(x)$ is therefore invertible, and continuous. Moreover, from the smoothness conditions on $h$, the first three derivatives of $F$ are continuous: $F'(x)=f(x)=h(F(x))$, $f'(x)=h'(F(x))f(x)=h'(F(x))h(F(x))$ and $f''(x)=h''(F(x))h^2(F(x))+h'^2(F(x))h(F(x))$. Therefore, the solution to the above ODE has the functional form $h(x)=f(F^{-1}(x))$, where $F$ has the required smoothness for It\^o's lemma to be used, and is invertible. It\^o's lemma yields the dynamics of $F^{-1}(Y_t)$, which corresponds to the SDE~(\ref{eq:LatProc0}) where $\eta(t,x)=g(t)$ and $\mu(t,x)=-\frac{g^2(t)}{2}\frac{f'(x)}{f(x)}$. From Theorem 4.5.3 of~\cite{Kloed99} (p. 131) this SDE, with finite initial value $X_0=F^{-1}(Y_0)$ has a strong pathwise unique solution since the coefficients meet the standard requirements (Lipschitz continuity of $\psi(x)$ together with the boundedness of $g^2(t)$ for $ t<\infty$ implies the linear growth bound condition on $g^2(t)\psi(x)$ and hence so is the drift coefficient $\mu(t,x)$). Finally, the solution $Y$ is given by the mapping $F$: $Y=F(X)$.
\end{proof}
\begin{example}
Consider the case of the exponential martingale with time
dependent volatility, with SDE $dY_t=\eta(t) Y_t dW_t$. Setting
$g(t)=\eta(t)$ and $h(x)=x$, we find $F(x)=\e^{x}+k$;
$\mu(t,x)=-\eta(t)^2/2$. In the case where $\eta(t)=\eta$, one
could equivalently choose $g(t)=1$ and $h(x)=\sigma x$, in which
case $F(x)=e^{\eta x} + k$ and $\mu(t,x)=-\eta/2$.
\end{example}
\begin{example}
This trick has been previously applied to the SDE~(\ref{eq:y1my})
in the case $g(t)=\eta$ and $h(x)=x(1-x)$. Similarly regarding
eq.~(\ref{eq:tanhSDE}), we can set $g(t)=\eta$ and
$h(x)=(1-x^2)/2$; the solution to the ODE~(\ref{eq:ODE1}) leads to
$F(x)=\tanh(x/2)$; therefore, the solution to~(\ref{eq:tanhSDE})
is given by $Y=\tanh(X/2)$ where $X$ is the solution
to~(\ref{eq:LatProc0}).
\end{example}
\begin{remark}
It is worth noting that although we obtain the SDE of
$X=F^{-1}(Y)$ from that of $Y$, the expression of $F^{-1}$ is not
needed; it does not enter the SDE of $X$. The drift of $X$ is
determined by the score function of $F$, which solves the ODE.
\end{remark}
%
\section{The $\Phi$-martingale}
\label{sec:Phi}
%

The previous computations done in  eq.~(\ref{Xgaus}) for $F=\Phi$
lead, for $\eta(t,x)=\eta$ to $dX_t= \frac{\eta^2}{2} X_tdt+\eta dW_t$, i.e. $X $
is a Vasicek process
\beq X_t=X_0\e^{\frac{\eta^2}{2}t}+\eta
\e^{\frac{\eta^2}{2}t}\int_{0}^t \e^{-\frac{\eta^2}{2}s}dW_s\label{eq:VasicekX} \eeq
with constant diffusion coefficient $\eta$, zero long-term mean
and negative speed of mean reversion $\eta^2/2$. Note that, for
fixed $t$, $X_t$ has the same law as \beq
X_0\e^{\frac{\eta^2}{2}t}+\sqrt{e^{\eta^2t}-1}Z  \eeq where $Z$ is
a standard Gaussian random variable.
This leads to a $[0,1]$-martingale $Y$ which analytical expression is $\Phi(X)$ where $X=(X_t)_{t\geq 0}$ is the Vasicek process~(\ref{eq:VasicekX}). The process $Y$ is called the $\Phi$-martingale. Sample paths drawn from this exact solution are shown in Fig.~\ref{fig:SamplePath}.
\begin{figure}
\begin{center}
\subfigure[$(Y_0,\eta)=(0.5,0.2)$]{\includegraphics[width=0.45\columnwidth]{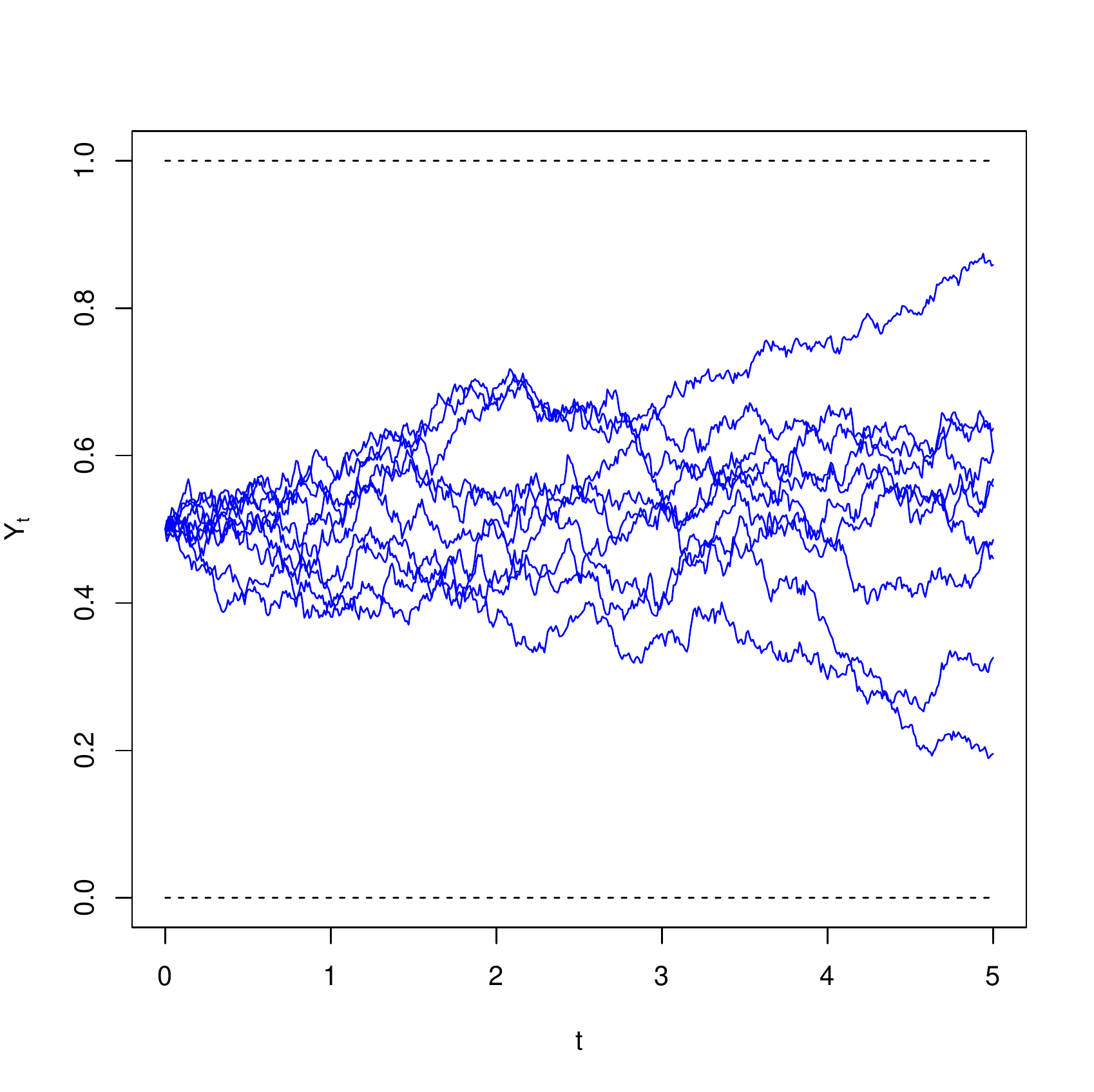}}\hspace{0.5cm}
\subfigure[$(Y_0,\eta)=(0.5,0.8)$]{\includegraphics[width=0.45\columnwidth]{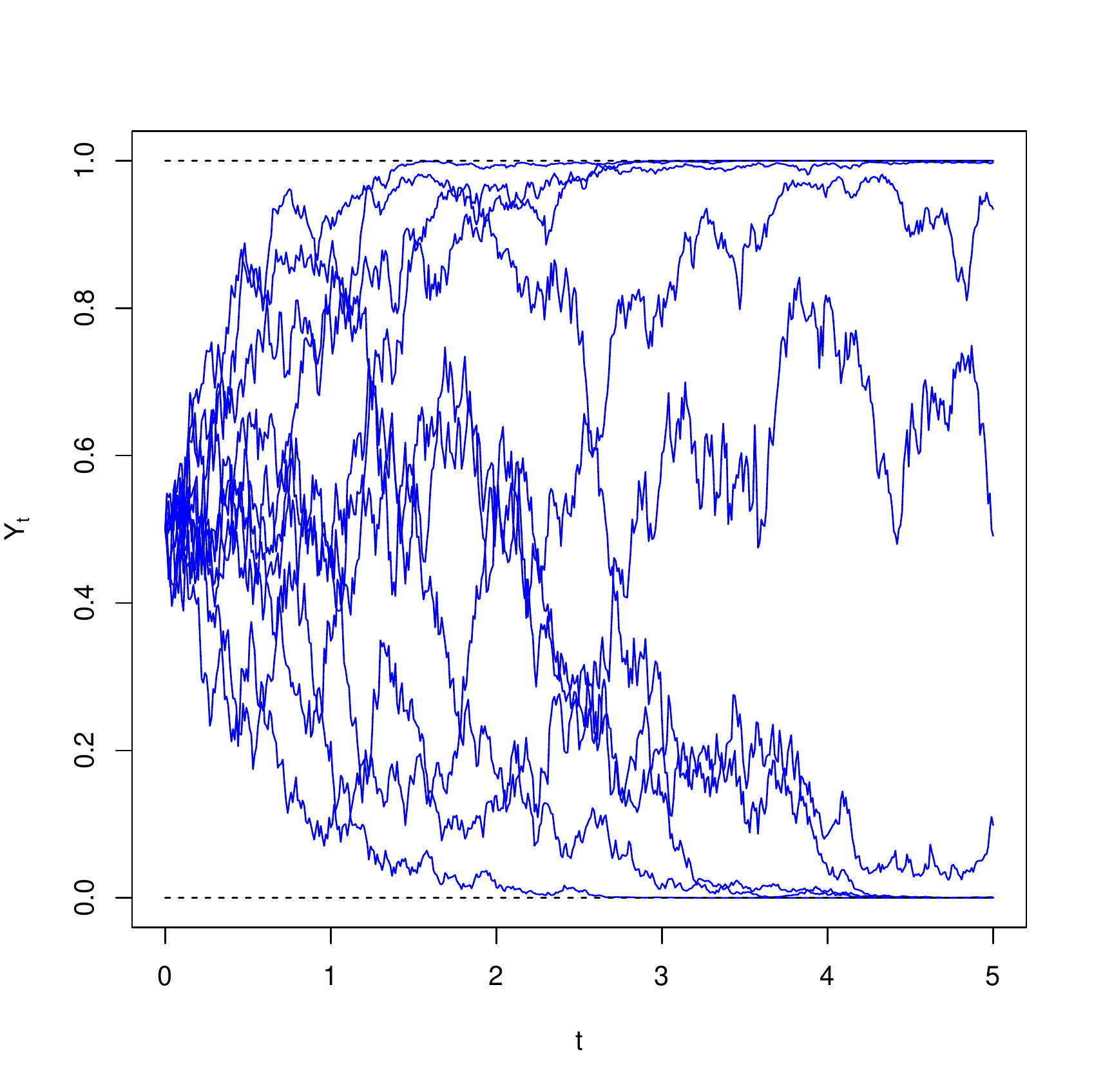}}\\
\subfigure[$(Y_0,\eta)=(0.75,0.2)$]{\includegraphics[width=0.45\columnwidth]{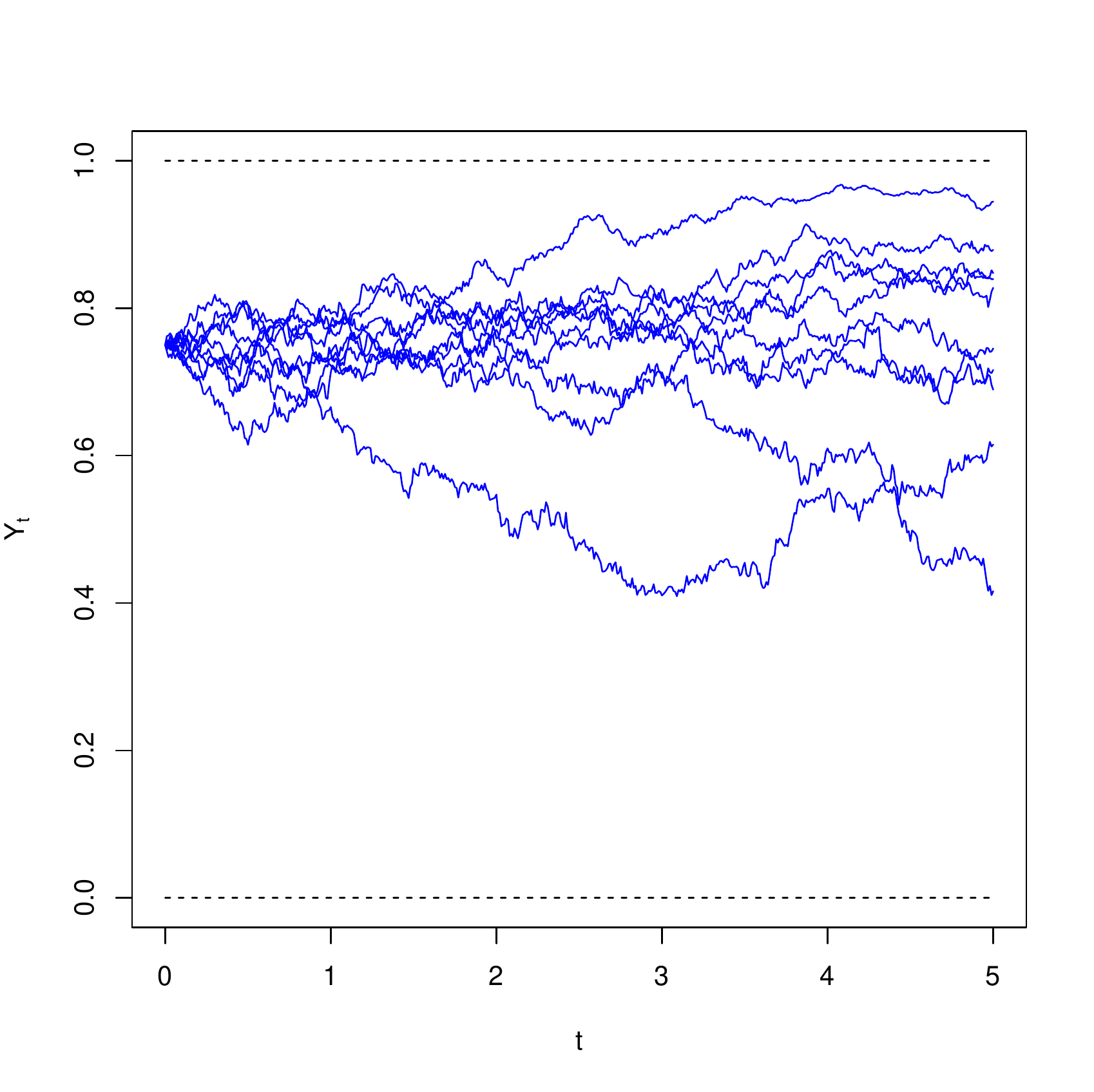}}\hspace{0.5cm}
\subfigure[$(Y_0,\eta)=(0.75,0.8)$]{\includegraphics[width=0.45\columnwidth]{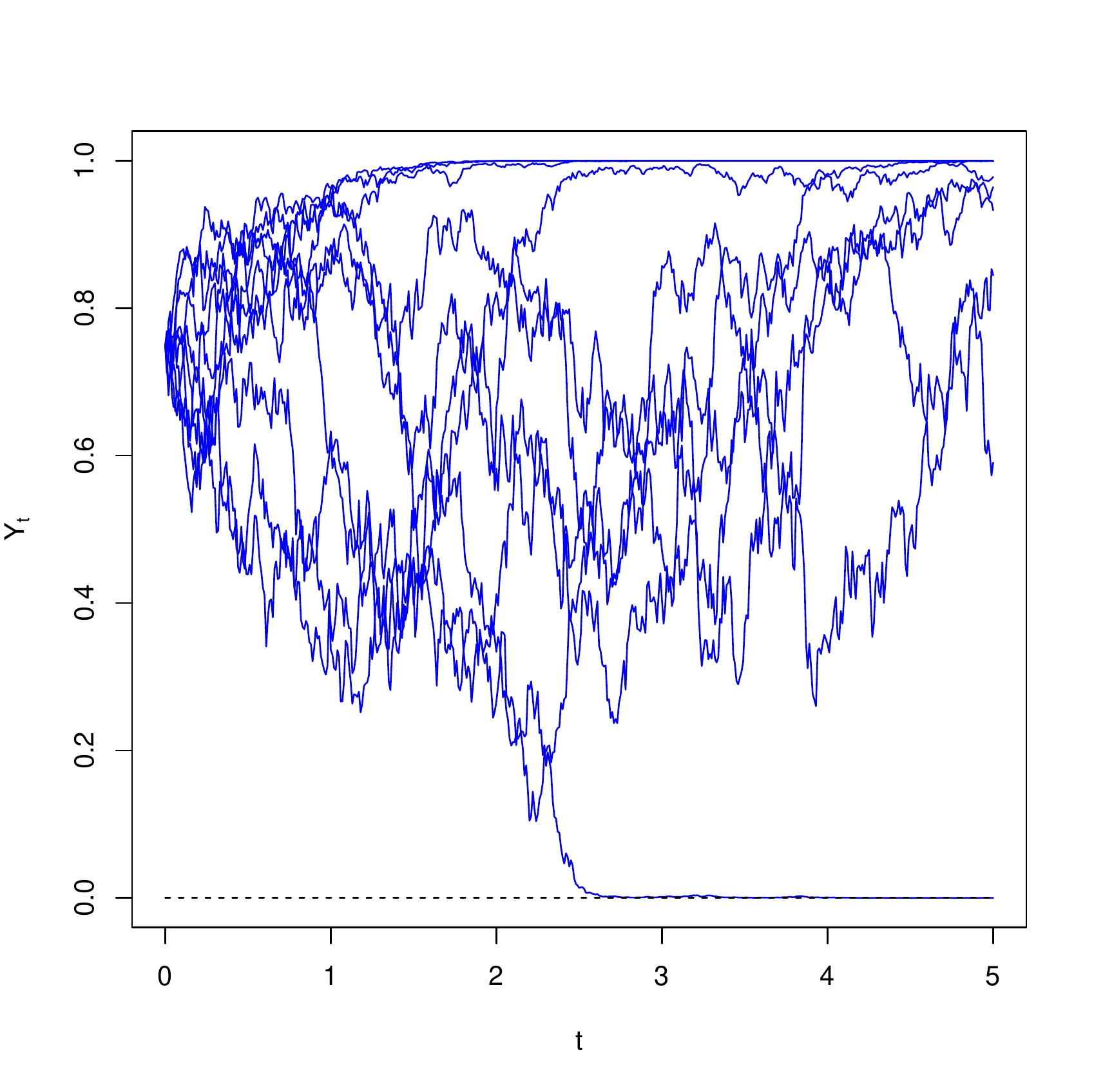}}
\caption{Ten sample paths of the $\Phi$-martingale $Y$ drawn from the exact solution $\Phi(X)$ for various diffusion scales $\eta$ and initial conditions $Y_0$.
}
\label{fig:SamplePath}
\end{center}
\end{figure}
\begin{remark}
It has been shown that the case~(\ref{eq:lnY}) is also tractable
since there is closed form expression for the latent process $X $ (and thus for $Y=\exp(-\lambda X)$). The joint density of $\left(\Theta_t,\int_{0}^t\Theta_sds\right)$ has been studied by Yor in~\citep{Yor92}, providing the law of $X_t$ (see Appendix). Howhever, the solution $X_t$ explodes at $\tau=\{t:\int_0^t\Theta_s ds = 2/(\lambda\eta^{2})\}$, and $\tau<\infty$ wp 1 as $\Theta_s$ is a grounded positive process. This means that the corresponding $Y_t=e^{-\lambda X_t}$, $\lambda>0$, will collapse to zero in finite time as well. This is not the case of the $\Phi$-martingale which merely \textit{asymptotically} collapses to the bounds, but belongs to $(0,1)$ $\mathbb{Q}$-a.s. for all $t>0$ (see Section~\ref{sub:asymptotics}).
\end{remark}

\subsection{Statistics and asymptotics}
\label{sec:distribution}
%
In the case where the SDE of the latent process $X$ (which drift is implied by $F$) has an
explicit solution, the process $Y$ can be studied in details.
For instance, the asymptotic distribution of $Y_t$ as $t\to\infty$
can be obtained. Moreover, one can also study the properties of
disjoint increments of $Y$. They have zero-mean and are
uncorrelated, as per the martingale property. Their variances and
quantile functions can be computed as well. We study below the statistics of the $\Phi$-martingale. For the sake of
comparison, we mention the corresponding results for the
exponential martingale.

In the case of $Y=\Phi(X)$ where $X $ is a Vasicek process with
instantaneous variance $\eta^2$, zero long-term mean and negative
mean reversion speed $\eta^2/2$, the variance of the random
variable $Y_t$ is given by $\E\left[Y_{t}^2\right]-Y_0^2$ where
\beqn
\E\left[Y_{t}^2\right]&=&\int_{0}^1\frac{y^2}{f(F^{-1}(y))}f_{X_t}(F^{-1}(y))dy \nonumber\\
&=&\Phi_2\left(X_0,X_0;\frac{\e^{\eta^2t}-1}{\e^{\eta^2t}}\right)\nonumber\\
&=&\Phi_2\left(X_0,X_0;1-\e^{-\eta^2t}\right)\nonumber 
\eeqn
where $\Phi_2(x,y;\rho)$ is the standard bivariate Normal
cumulative distribution with correlation $\rho$. In particular,
$\lim_{t\to\infty}
\E\left[Y_{t}^2\right]-Y_0^2=\Phi(X_0)-Y_0^2=Y_0(1-Y_0)$.

As per properties of martingales
$\E\left[Y_{s}Y_{t}\right]=\E\left[Y^2_{s\wedge t}\right]$ so that for any $\delta\geq 0$, the
auto-covariance of $\{Y_t,Y_{t+\delta}\}$ is equal to the variance
of $Y_t$. The variance of the increments is then given by
\beqn
\var\left[Y_{t+\delta}-Y_t\right]&=&\var\left[Y_{t+\delta}\right]-\var\left[Y_{t}\right] \nonumber\\
&=&\E\left[Y_{t+\delta}^2\right]-\E\left[Y_{t}^2\right]\nonumber\\
&=&\Phi_2\left(X_0,X_0;1-\e^{-\eta^2(t+\delta)}\right)  -\Phi_2\left(X_0,X_0;1-\e^{-\eta^2t}\right)\nonumber
\eeqn

which converges to zero as $t\to\infty$. Intuitively, this means
that the ``activity'' of the process (path by path) will decrease
with time, and the process will converge to some constant level.
By comparison, the variance of the exponential martingale $M_t=M_0\e^{-\eta^2/2t+\eta W_t}$ increases
with $t$:
\beqn
\var\left[M_{t}-M_{s}\right]&=&\E\left[\left(M_{t}-M_{s}\right)^2\right]=\E\left[M_{s}^2\left(\frac{M_{t}}{M_{s}}-1\right)^2\right]\\
&=&\E\left[M_{s}^2\right]\left(\E\left[\frac{M_{t}^2}{M_{s}^2}\right]-2\E\left[\frac{M_{t}}{M_{s}}\right]+1\right)\nonumber\\
&=&X_0\e^{\eta^2s}\E\left[\e^{-(2\eta)^2/2s+2\eta\sqrt{s}Z}\right]\times\left(\e^{\eta^2(t-s)}\E\left[\e^{-(2\eta)^2/2(t-s)+2\eta\sqrt{t-s}Z}\right]\right.\\
&&~~~~~~~~~~~\left.-2\E\left[\e^{-\eta^2/2(t-s)+\eta\sqrt{t-s}Z}\right]+1\right)\nonumber\\
&=&X_0\e^{\eta^2s}\left(\e^{\eta^2(t-s)}-1\right)\nonumber \eeqn
Because the paths of the $\Phi$-martingale $Y$ evolve between two bounds, a central question is to determine whether they collapse to the bounds, in which case the distribution of $Y_t$ would have less and less mass in $(0,1)$ in the sense that for any arbitrarily small threshold $\epsilon>0$ and any probability level $0<p<1$, there exists a time $t_{\epsilon}$ such that for all $t>t_\epsilon$, $\Q\{Y_t\in[0,\epsilon)\cup(1-\epsilon,1]\}>p$; $Y_t$ ends up in the neighborhood of the bounds with any desired confidence interval. This will be proven in the case when $F=\Phi$ and $\eta(t)=\eta$ (Section~\ref{sub:asymptotics}). An intuitive development is provided in Appendix~(\ref{app:collapse}) in the more general case of bounded martingales.

This might be an argument to show that this specific setup is not appropriate in many cases. However, this distribution behavior is shared by the quite popular geometric Brownian motion for example. The distribution of the exponential martingale is collapsing to 0 as $t\to\infty$ (see Fig.~\ref{fig:CLNQ} for an illustration of the quantiles for the corresponding stochastic processes $M_t$, $q(t,p):= q:\Q\{M_t\leq q\}=p$).
The fact that for this process, the variance of $M_{t+\delta}-M_t$ is not converging to zero as time passes in spite of this collapsing feature results from the fact that the right tail of the exponential martingale distribution is unbounded.

\subsubsection{Asymptotic distribution of the exponential martingale}
%
Consider the exponential martingale $M$ introduced above. The corresponding quantile function $q(t,p)$ defined according to $\Q\{[M_t\leq q(t,p)\}=p$ is
\beq q(t,p)=\exp\left(\eta\sqrt{t}\Phi^{-1}(p)-\eta^2/2t\right)
\nonumber \eeq
and for $0<p<1$,
\beq \lim_{t\to\infty} q(t,p) = \lim_{t\to\infty}
\exp\left(\sqrt{t}(a-b\sqrt{t})\right)\nonumber \eeq
for some finite $-\infty<a:=\eta\Phi^{-1}(p)<\infty$ and $0<b:=\eta^2/2<\infty$. For $t\geq t^\star:=\max(0,a/b)$, the expression in the RHS limit is strictly decreasing to 0 with respect to $t$. The $p=50\%$ case (median) is precisely the largest $p$ such that the curve is decreasing everywhere.\footnote{This contradicts the naive interpretation of martingales having ``no tendency to raise or fall''; the exponential martingale  \textit{does} have a tendency to fall since $\Q\{M_{t+\delta}<M_t\}>50\%$ but its expectation does not $\E[M_{t+\delta}]=\E[M_t]$. This reflects the fact that the martingale $M$ satisfies $\lim_{t\rightarrow \infty } M_t=0$} 
\begin{figure}
\begin{center}
\subfigure[Exponential martingale with $M_0=1$]{\includegraphics[width=0.45\columnwidth]{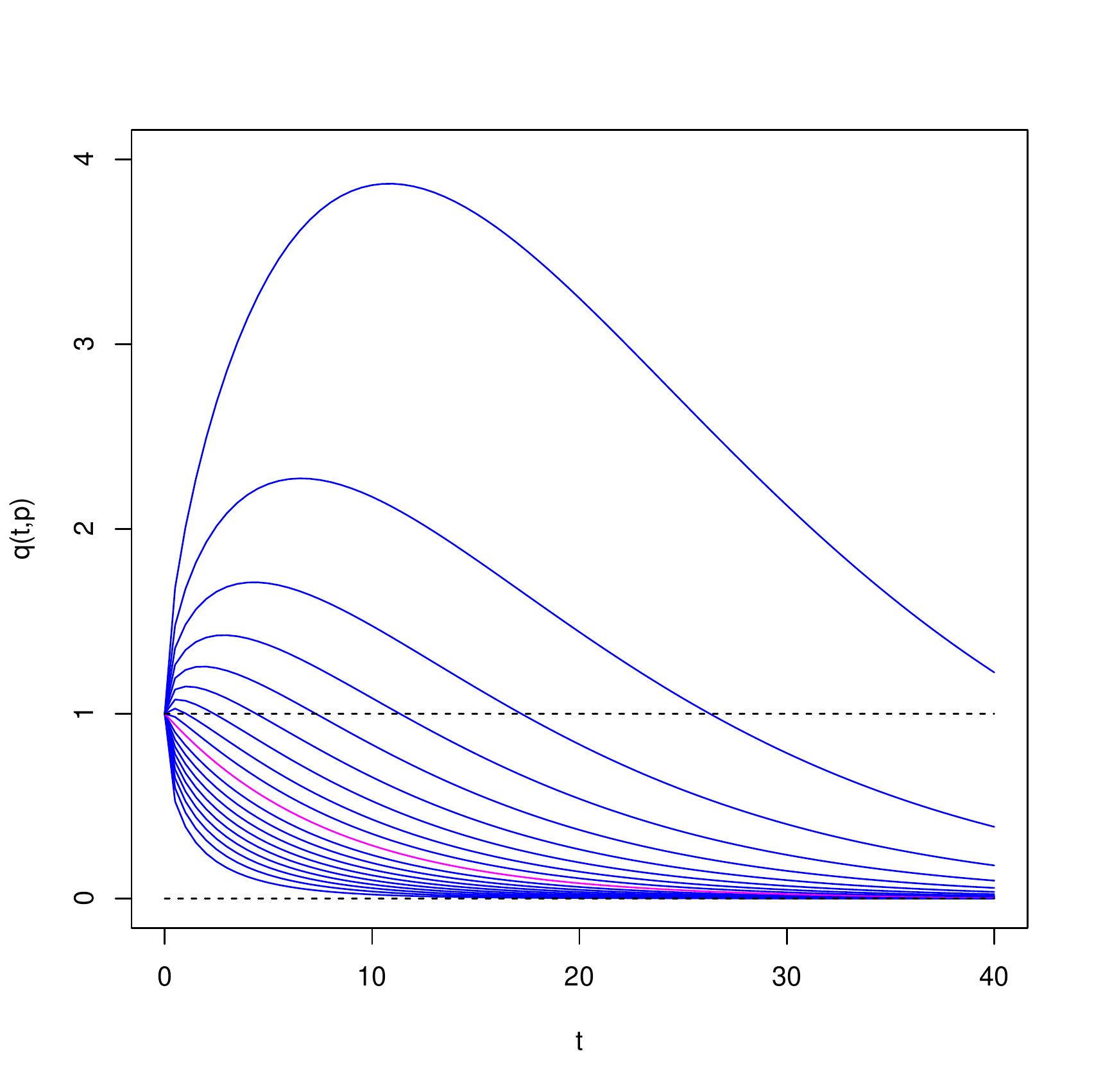}}\hspace{0.5cm}
\subfigure[$\Phi$-martingale with $Y_0=0.5$]{\includegraphics[width=0.45\columnwidth]{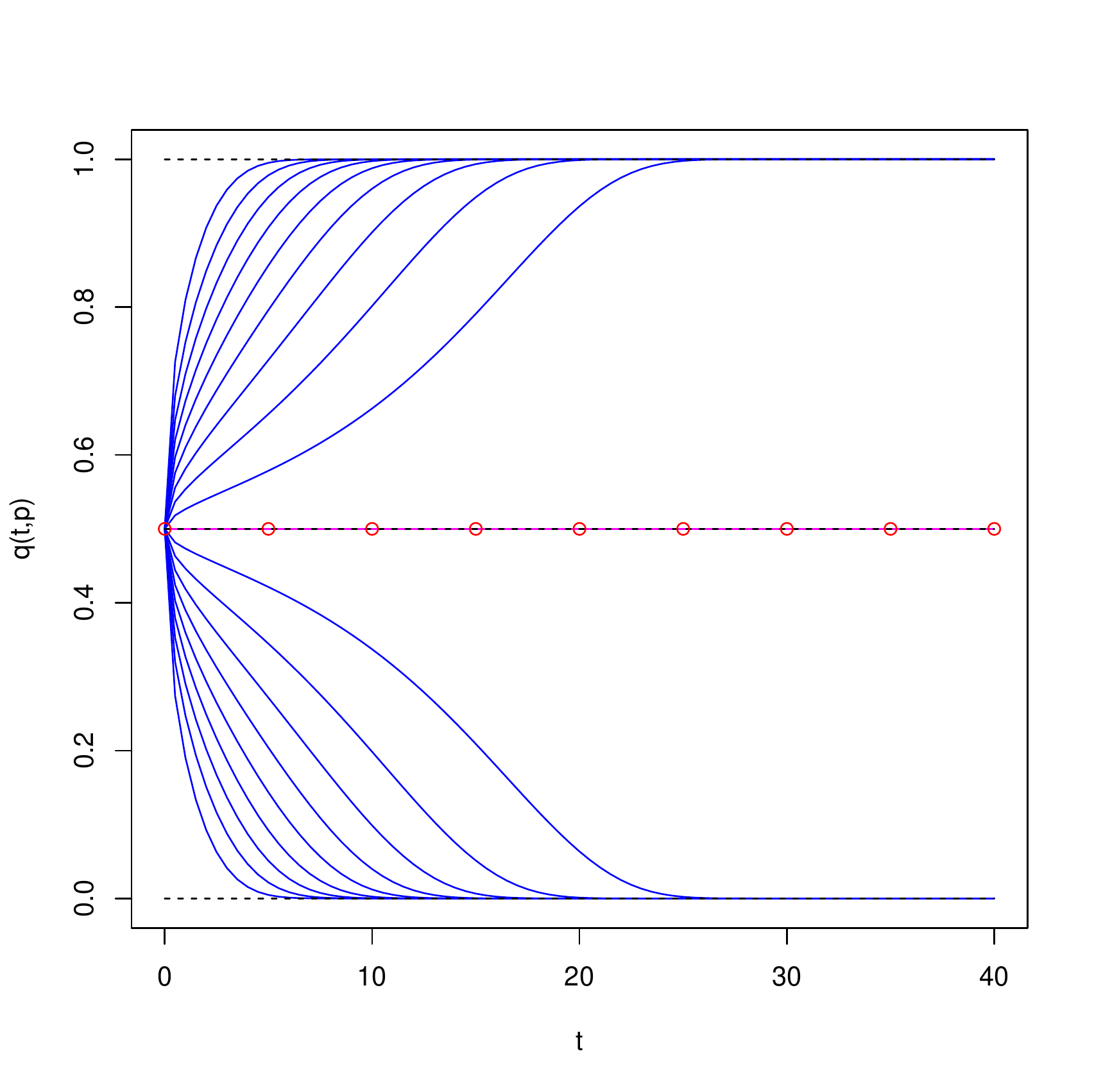}}\\
\subfigure[$\Phi$-martingale with $Y_0=0.4$]{\includegraphics[width=0.45\columnwidth]{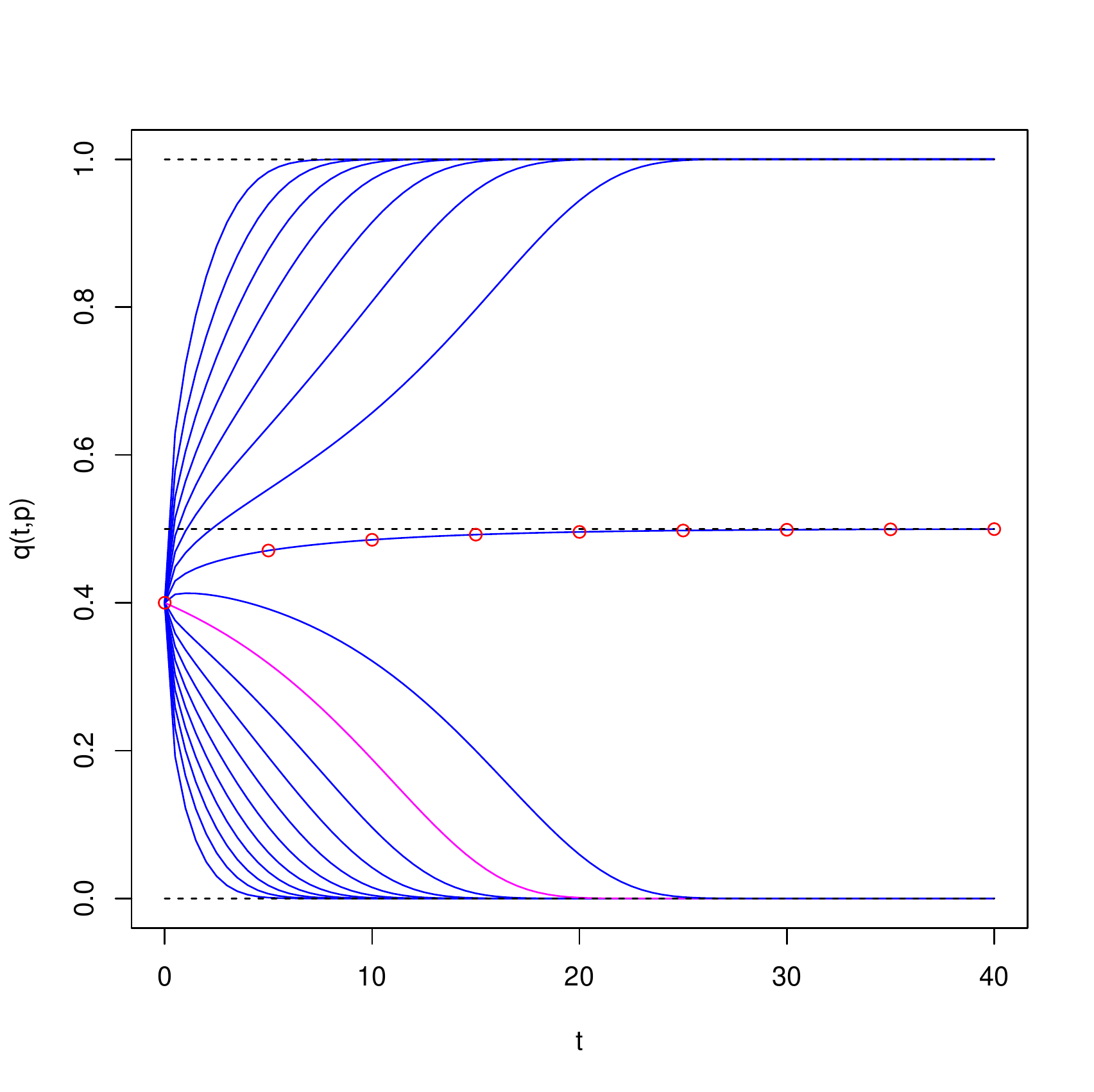}}\hspace{0.5cm}
\subfigure[$\Phi$-martingale with $Y_0=0.6$]{\includegraphics[width=0.45\columnwidth]{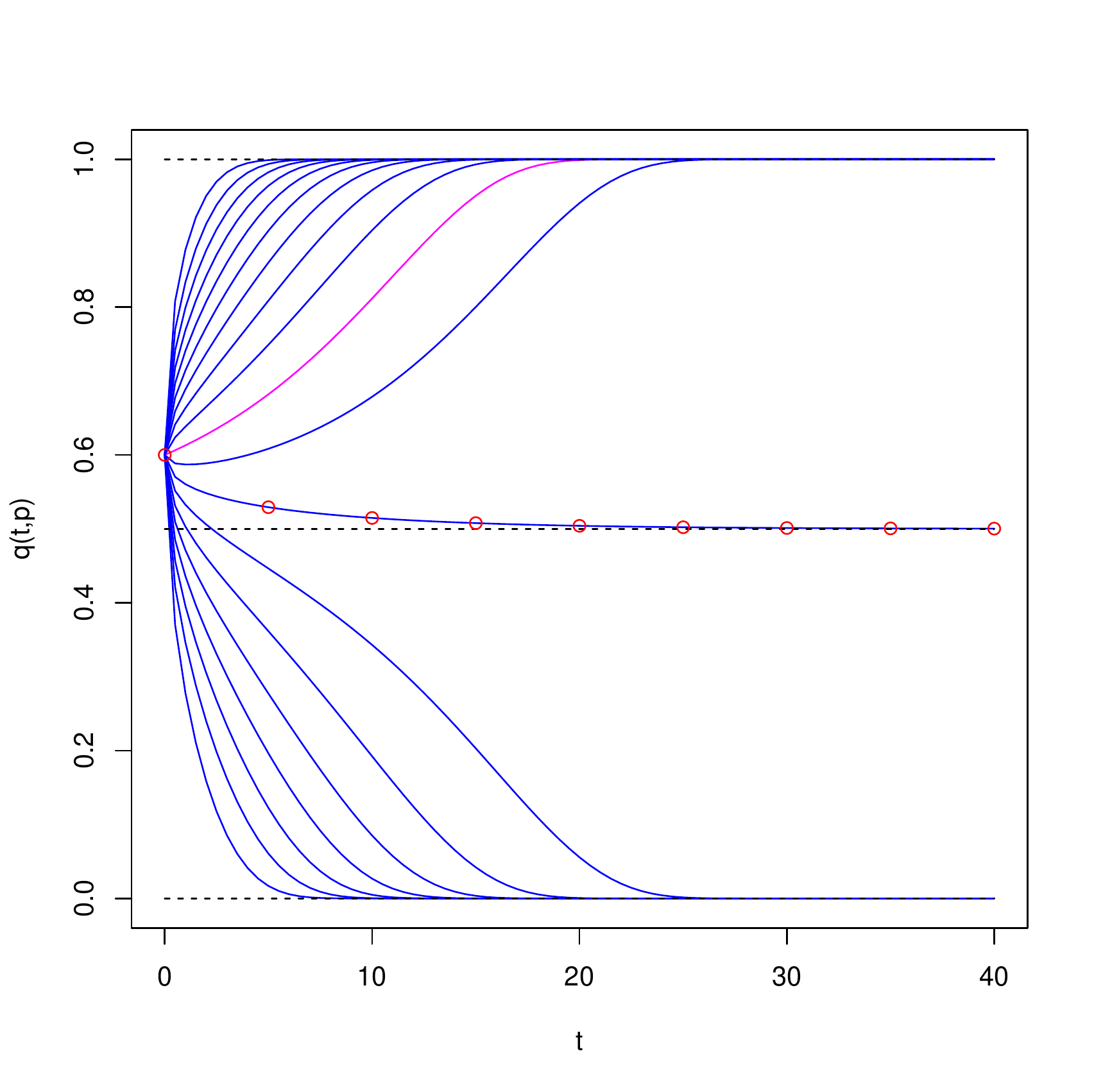}}
\caption{Quantile trajectories $q(t,p)$ of the distributions of the exponential martingale $M_t$ and the $\Phi$-martingale $Y_t$ for $\eta=50\%$. The curves are shown for $p\in\{5\%,10\%,\ldots,95\%\}$ probability levels (of course, the corresponding trajectories are ordered bottom up). For the exponential martingale case, the curve associated to the median is the largest decreasing curve. For the bounded martingale case, the distribution collapses to a $Bernoulli(Y_0)$. The median is shown in magenta, and the $(1-Y_0)$-quantiles are emphasized with dots. If $Y_0=0.5$, the distribution is equally splitted to the bounds. }
\label{fig:CLNQ}
\end{center}
\end{figure}
%
\subsubsection{Asymptotic distribution of the $\Phi$-martingale $Y $ ($\eta(t)=\eta$)}
\label{sub:asymptotics}
%
We get
\beqn
\lim_{t\to\infty} \Q\{Y_t\leq y\}&=&\lim_{t\to\infty} \Phi\left(\Phi^{-1}(y),\Phi^{-1}(Y_0)\e^{\eta^2t/2},\sqrt{\e^{\eta^2t}-1}\right)\nonumber\\
&=&\lim_{t\to\infty} \Phi\left(\frac {\Phi^{-1}(y)-\Phi^{-1}(Y_0)\e^{\eta^2t/2}}{\sqrt{\e^{\eta^2t}-1}}\right)\nonumber\\
&=&0\ind_{\{y=0\}}+\Phi\left(-\Phi^{-1}(Y_0)\right)\ind_{\{0<y<1\}}+\ind_{\{y=1\}}\nonumber\\
&=& (1-Y_0) \ind_{\{0<y<1\}}+\ind_{\{y=1\}}\nonumber
\eeqn

Therefore, $Y_t$ converges in distribution to $Bernoulli\left(Y_0\right)$ as $t\to\infty$. It is worth noting that $Bernoulli\left(Y_0\right)$ corresponds to the distribution in $[0,1]$ with maximum variance for a given mean $Y_0\in[0,1]$ (this is quite intuitive and easy to prove).

Because the Gaussian solution $X$ does not explode, the collapsing feature of the $\Phi$-martingale is an asymptotic behavior: $\Q\{Y_t\in\{0,1\}\}=0$ for all $t>0$. This is in contrast with the case~(\ref{eq:ExpSDE}) with $\mu=\lambda\eta^2/2>0$ where the process $Y=\exp(-\lambda X)$ has a positive probability to be strictly zero before any finite time: $\forall t>0,~\Q\{Y_t=0\}>0$.

\subsection{Autonomous Gaussian martingales}

We are investigating in which case a continuous local martingale $Y $ which is a diffusion  with separable diffusion coefficient can be written as the time-homogeneous (that is, autonomous) mapping $F$ of a  Gaussian diffusion $X $.
\begin{definition}[Gaussian Diffusion] A \emph{Gaussian diffusion} is the unique solution $X$ to the SDE~(\ref{eq:LatProc00}) where the drift $\mu(t,x)$ is affine in $x$, $\mu(t,x)=a(t)+b(t)x$ and the diffusion coefficient $\eta(t,x)$ is a function of time only, $\eta(t,x)=\gamma(t)<\infty$ for all $t$.
\end{definition}
Observe that not all Gaussian processes are Gaussian diffusions in the sense of the above definition. For instance, the solution $X$ to the SDE $dX_t=\sign(W_t)dW_t$ is not a Gaussian diffusion but is a Brownian motion (and thus a Gaussian process),   and Fractional Brownian motions are Gaussian processes which are even not  semi-martingales.
\begin{definition}[Autonomous Gaussian Martingales] We say that the martingale $Y$ is \emph{autonomous Gaussian} if (i) it can be obtained by mapping a Gaussian diffusion $X$ through an autonomous mapping $F(x)$ and (ii) the diffusion coefficient is separable in the sense of (\ref{eq:SDEySep}).
\end{definition}
Equating the $(dt)$ and $(dW_t)$ terms of the $Y$ SDE~(\ref{eq1}) with that of the $F(t,X_t)$ SDE obtained using It\^o, we get
\beqn
(a(t)+b(t)x)F_x(t,x) + \frac{\eta^2(t)}{2}F_{xx}(t,x)&\stackrel{(dt)}{=}&-F_t(t,x) \nonumber\\
\gamma(t)F_x(t,x) &\stackrel{(dW_t)}{=}& \eta(t) h\circ F(t,x) \nonumber
\eeqn
Using a time-homogeneous mapping yields $F(t,x)=F(x)$, implying that $\eta(t)=\gamma(t)$. The $(dW_t)$ equation then corresponds to eq.~(\ref{eq:ODE1}) and it solution yields the space component $h(y)$ of the separable diffusion coefficient $\sigma(t,y)=\gamma(t)h(y)$. In this exercise however, we are interested in the form of $F(x)$ that can be used so that $X $ is a Gaussian process and $Y =F(X )$ a martingale. Setting $G(x)=F_x(x)$, the $(dt)$ equation becomes
\beq
\frac{\eta^2(t)}{2}G_x(x) + (a(t)+b(t)x)G(x) =0 \nonumber
\eeq
This is a first-order ODE which solution is given by
\beqn
G(t,x) &=&k_1(t)\e^{-\frac{2a(t)x+b(t)x^2+c(t)}{\eta^2(t)}}\nonumber\\
F(t,x) &=&k_1(t)\int_{-\infty}^x\e^{-\frac{2a(t)u+b(t)u^2+c(t)}{\eta^2(t)}}du+k_2(t)\label{eq:GenSolSDE}
\eeqn
As we considered time-independent mapping ($F(t,x)=F(x)$ for all $x$), the ratios $a(t)/\eta^2(t)$, $b(t)/\eta^2(t)$, $c(t)/\eta^2(t)$ and the integration constants $k_1(t),k_2(t)$ need all to be constant in order to get the required form for the $Y_t=F(X_t)$ SDE. We note them $a,b,c,k_1,k_2$. Clearly, the case $b<0$ can be excluded as the integral in eq.~(\ref{eq:GenSolSDE}) does not converge in this case. We thus have three main cases for $(a,b,c)$ to analyze:

\begin{itemize}
\item  $(0,0,c)$: $F(x)=k_1\e^{-c}x+k_2$; the mapping is an affine function of the form $F(x)=\alpha x +\beta$.
\item  $(a< 0,0,c)$: $F(x)=\frac{-k_1\e^{-c}}{2a}\e^{-2ax}+k_2$, the mapping is a shifted exponential $F(x)=\frac{-k_1\e^{-c}}{2a}\e^{-2ax}+k_2=\frac{\alpha}{\xi}\e^{\xi x}+k$ with $\xi>0$.
\item  $(a,b>0,c)$: $F(x)=\alpha\Phi(\frac{x-\beta}{\xi})+k_2$ where $\alpha=k_1\sqrt{\pi/b}\e^{a^2/b-c}>0$, $\beta=-a/b$ and $\xi=1/\sqrt{2b}$. The mapping is a shifted and rescaled version of the Normal cumulative distribution function.
\end{itemize}

The above mappings are the only ones leading to null $dt$ term and separable diffusion coefficient for $Y_t=F(X_t)$ when $X_t$ is a Gaussian diffusion. This also specify the form of the space component $h(x)$ of the diffusion coefficient that can be obtained by mapping Gaussian processes through $F(x)$. From the $(dW_t)$ equation, we get respectively

\begin{itemize}
\item Since $a=b=0$, $X$ is a rescaled Brownian motion ($d\langle X,X\rangle_t=\gamma^2(t)dt$) and $Y_t=\alpha X_t+k$ is a shifted and rescaled copy. In particular, $h(x)=h=\alpha$
\item $Y$ is the exponential of a Gaussian process with shift: $F_x(x)=\alpha\e^{\xi x}$ so that $h(x)=F_x(F^{-1}(x))=\xi(x-k)$.
Since $\xi>0$, a continuity argument shows that the process $Y$ is bounded below $k$ if $Y_0>k$.
\item $Y$ is obtained by mapping the Gaussian process $X$ through a Normal cumulative distribution, rescaling and shifting. The obtained process is a $[k,\alpha+k]$-martingale. In this case, $h(x)=\frac{\alpha}{\beta}\phi\left(\Phi^{-1}\left(\frac{x}{\alpha}\right)\right)$.
\end{itemize}

We summarize these results in the theorem below.
\begin{theorem}[Autonomous Gaussian Martingales] The only autonomous Gaussian martingales are (up to a deterministic shift and scaling coefficient) i) the trivial martingale, ii) the Brownian motion, iii) the geometric Brownian motion and iv) the $\Phi$-martingale. Interestingly, each resulting process has a specific range, namely: constant, unbounded, one-side bounded and two-sides bounded.
\end{theorem}
This result says that if one wishes to construct a continuous local martingale $Y$ with separable diffusion coefficient $\sigma(t,y)=\eta(t)h(y)$ and evolving in a given set by mapping a Gaussian diffusion via an invertible autonomous function, there are not many alternatives: only one family of mapping per type of range. In particular, the $\Phi$-martingale $\Phi(X)$ is the only bounded continuous martingale with separable diffusion coefficient that can be obtained by mapping a Gaussian diffusion $X$ through a smooth autonomous function $F(x)$. 

Note that if one relaxes the time-homogeneity and invertibility
constraints, other solutions are possible. For instance, in the
case $a(t)=b(t)=0$ and $\eta(t)=\eta$ ($X_t=X_0+\eta W_t$) and
setting $F(t,x)=x^2-\eta^2 t$, we obtain $Y_t=X^2_t-[X]_t$ which is a well known martingale (in $\mathbb{R}$).

In the next section, we show how these martingales can be used in credit risk modeling applications.

\section{Application to Survival Probabilities}
\label{sec:SP}

We adopt the credit risk modeling setup and focus on the default
time $\tau$ of some reference entity. In this framework, one
usually defines the filtration $\mathbb{F}$, which represents the
market information excluding default observation. The enlarged
filtration is obtained by including the explicit information
relevant to the default event: $\mathcal{G}_t=\filF_t\vee
\sigma(\ind_{\{\tau>s\}},0\leq s\leq t)$ (with right-continuous
regularisation) . In Cox models for example, the stochastic
intensity process $\lambda $ is $\mathbb{F}$-adapted, but
conditional upon the path $(\lambda_t)_{t\geq 0}$, the occurrence
of default $\ind_{\{\tau\leq t\}}$ is an independent event. More
generally, the latter is $\mathcal{G}_t$-measurable, but not
$\filF_t$-measurable. Literature on credit risk modeling emphasize
that under some conditions, one can get rid of actual default
modeling; default indicators can be replaced by stochastic default
probabilities, working in the restricted filtration $\mathbb{F}$
instead of the complete filtration $\mathbb{G}$. We do not enter
the details of this modeling approach, but refer the reader
to~\cite{Lando04} and~\cite{Biel11} for more information.

\subsection{Unconditional survival probability and Az\'ema supermartingale}
%
We now move to the modeling of the martingale $S_{t}(T), t\geq 0$
defined in eq.~(\ref{eq:DefStT}). This is useful in many
circumstances, including the pricing of credit derivatives or to
adjust the price of a derivatives portfolio to account for
counterparty risk (credit value adjustment), see
e.g~\cite{Ces09},~\cite{Brigo05}. As an
illustration of the above methodology, we set
$S_{t}(T):=\Phi(X_{t}(T))$ where $X_{t}(T)$ satisfies
\beq
dX_t(T)=(\eta^2/2) X_ {t}(T)dt+\eta dW_t
\eeq

Clearly, $S_{t}(T), t\geq 0$ is a martingale with initial value $S_0(T)$.

The initial survival probability function $S_0(t)$ is assumed to
be provided (in credit derivative applications, it is obtained by
bootstrapping market quotes of financial instruments, like
defaultable bonds or credit default swaps). It is decreasing and
satisfies for all $t>0,$ $0<S_0(t)<1$, which means the the
associated hazard rate is strictly positive and finite. This leads
to
\beqn
S_{t}(T)&=&\Phi\left(m(t,T)+\eta Z_t\right)\\
Z_{t}&:=&\int_{0}^t e^{(\eta^2/2) (t-s)}dW_s\label{eq:Zt}\\
m(t,T)&:=&X_{0}(T)\e^{(\eta^2/2) t} \nonumber\eeqn
The Az\'ema supermartingale $S_{t}:=S_{t}(t), t\geq 0$ is often
modeled either by using a simple Gaussian process ~\cite{Ces09} or
by adopting the Cox process \cite{Brigo05}. However, the first
approach clearly violates the $[0,1]$ condition, and the second
corresponds to the specific case where $S $ is a decreasing
predictable process. However, the general Doob-Meyer decomposition
(in its additive form) reveals that in all generality, we
have~\cite{Biel11},~\cite{Prof06}
\beq dS_t=dD_t+dM_t \nonumber\eeq
where $D $ is a decreasing $\mathbb F$-predictable process and $M$
a martingale. The Cox setup is just one particular case. Moreover,
it is hard to find a positive stochastic intensity model which
allows for analytical calibration to market quotes whilst
preventing negative path (in particular, square-root diffusion
processes need to be shifted for calibration purposes, so that the
resulting intensity process may not be positive anymore). This
gives room to alternative modeling setups, and the
$\Phi$-martingale is one of them~\cite{Vrins14}. The dynamics of
the associated Az\'ema's  martingale is proven to be
\beqn
dS_t&=&\zeta_tdS_0(t)+\eta\phi(\Phi^{-1}(S_t))dW_t\label{eq:StSDE}\\
\zeta_t&:=&\frac{\phi(\Phi^{-1}(S_t))\e^{\eta^2/2
t}}{\phi(\Phi^{-1}(S_0(t)))} \nonumber\eeqn
\begin{figure}
\centering
\subfigure[Density of $S_T$. Histogram (10k paths, Euler discretization) and theoretical density obtained by differentiating $F_{Y_t}$ in~(\ref{eq:CDFtransform}) with $F=\Phi$ and $X_t\sim \mathcal{N}\left(m(t, T) +\eta Z_t\right)$ (red)]{\includegraphics[width=0.45\columnwidth]{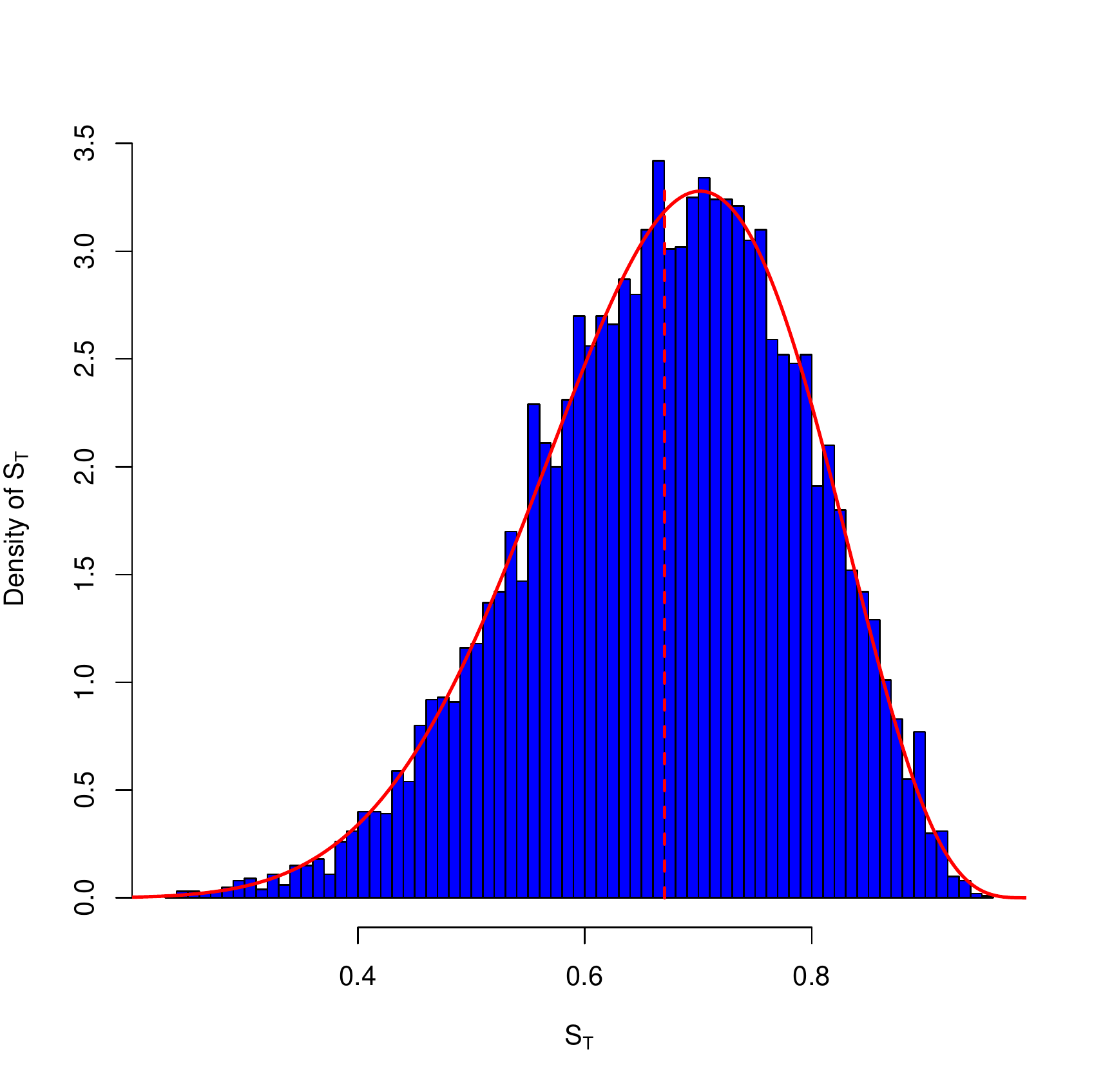}}
\hspace{0.5cm}
\subfigure[Sample paths of the Az\'ema supermartingale $S_t$ using the analytical solution $\Phi(X_t)$ (cyan) and Euler discretization of $S_t$ SDE~(\ref{eq:StSDE}) (dark blue)]{\includegraphics[width=0.45\columnwidth]{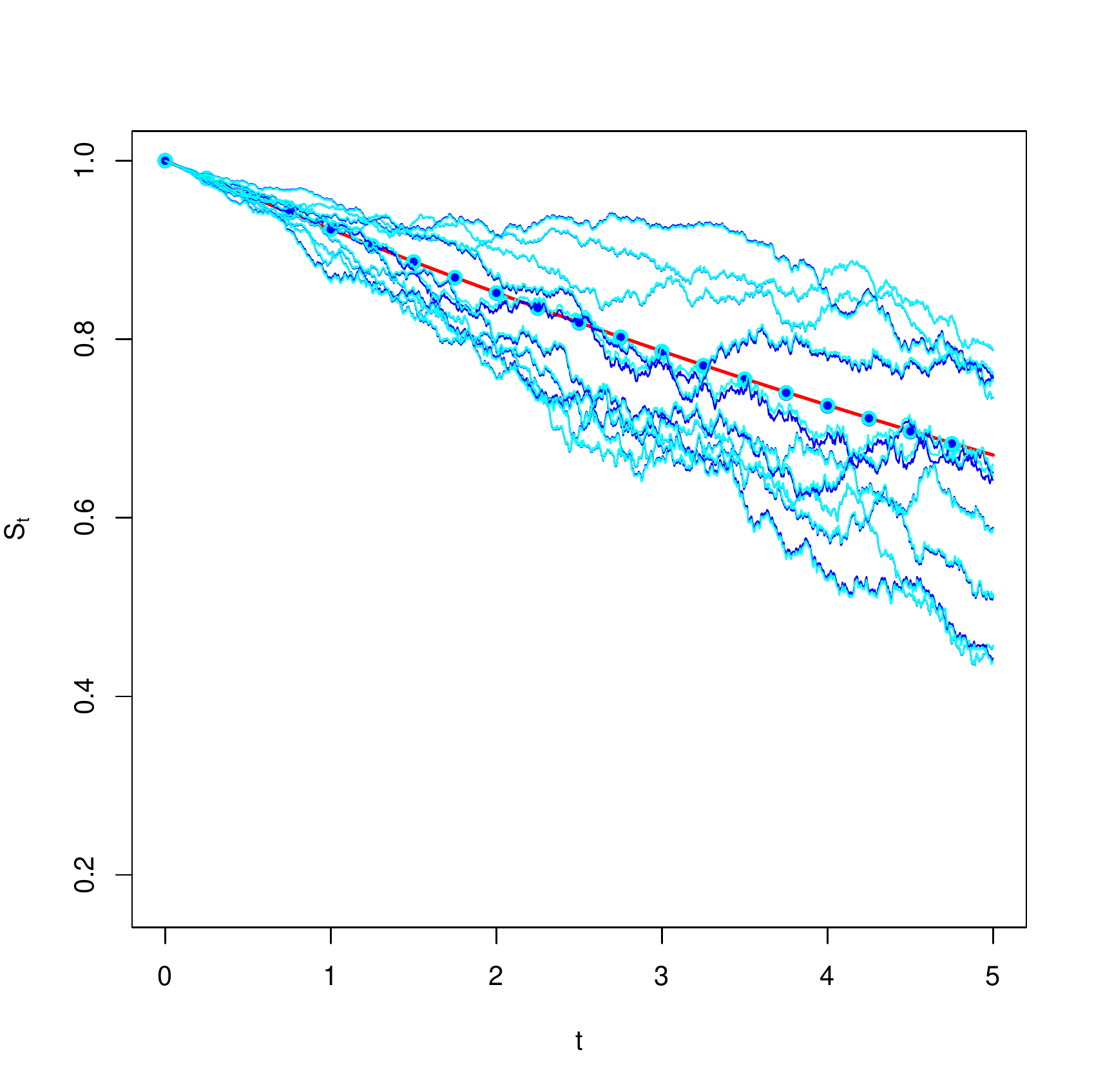}}
\caption{Distribution of $S_T$ and sample paths of $S_t$. Parameters: $\eta=0.15,~S_0(t)=\e^{-ht},~h=8\%,~T=5$.}\label{fig:St}
\end{figure}

\subsection{Unconditional survival probability and
Az\'ema supermartingale}

The \textit{survival probability up to time $T\geq t$ given no default prior to time $t$} is obtained from Bayes' rule:
\beq
Q_{t}(T):=\E\left[\ind_{\{\tau>T\}}\Big|\mathcal{F}_t,\{\tau>t\}\right]=\frac{\E\left[\ind_{\{\tau>(T\vee
t)\}}\Big|\mathcal{F}_t\right]}{\E\left[\ind_{\{\tau>t\}}\Big|\mathcal{F}_t\right]}=\frac{S_{t}(T)}{S_{t}(t)}
\eeq
which belongs to $[0,1]$ almost surely and is decreasing with respect to $T$ for all $T\geq t$.

We illustrate in Fig.~\ref{fig:QtT} the distribution of $Q_{t}(T)$ for the 16 nodes $z_i$ associated to the 16-points Gauss-Hermite quadrature associated to the standard Normal factor $Z^\star_t=\eta Z_t/\sqrt{v(t)}$ where $Z_t$ is given by eq.~(\ref{eq:Zt}) and $v(t)=\e^{\eta^2 t}-1$ is the variance of $\eta Z_t$. If $Q(t,T;z)$ stands for the value $Q_t(T)$ conditional upon $Z^\star_t=z$ and $(\omega_i,z_i)$, $i\in\{1,2,\ldots,n\}$ are the weights and nodes of the $n$-points Gauss-Hermite quadrature, then
\beqn
\Q_0\left\{\tau>T|\tau>t\right\}&=&\E[Q_{t}(T)]\nonumber\\
&=&\E[Q(t,T,Z^\star_t)]\nonumber\\
&\approx& \sum_{i=1}^n \omega_i Q(t,T;z_i)\nonumber\\
Q(t,T;z)&:=&\frac{\Phi\left(m(t,T)+\sqrt{v(t)}z\right)}{\Phi\left(m(t,t)+\sqrt{v(t)}z\right)}\stackrel{(\eta\to0)}{\to}\frac{S_0(T)}{S_0(t)}\nonumber
\eeqn
We consider the $t=0$ survival probability curve $S_0(t)=\e^{-\int_{0}^t h(s)ds}$ where the piece-wise constant hazard rate function $h(t)$ is given by the step function $\{t,\gamma(t)\}=\{(1,5\%),(3,6\%),(5,8\%),(7,8.5\%),(10,6.5\%)\}$. One can see that the implied $Q(t,T;z_i)$ curves can be quite different depending on the value of the driven factor, provided that the diffusion parameter $\eta$ of the Gaussian process underlying the $\Phi$-martingale is large enough.
\begin{figure}
\centering
\subfigure[$\eta=0.1$]{\includegraphics[width=0.45\columnwidth]{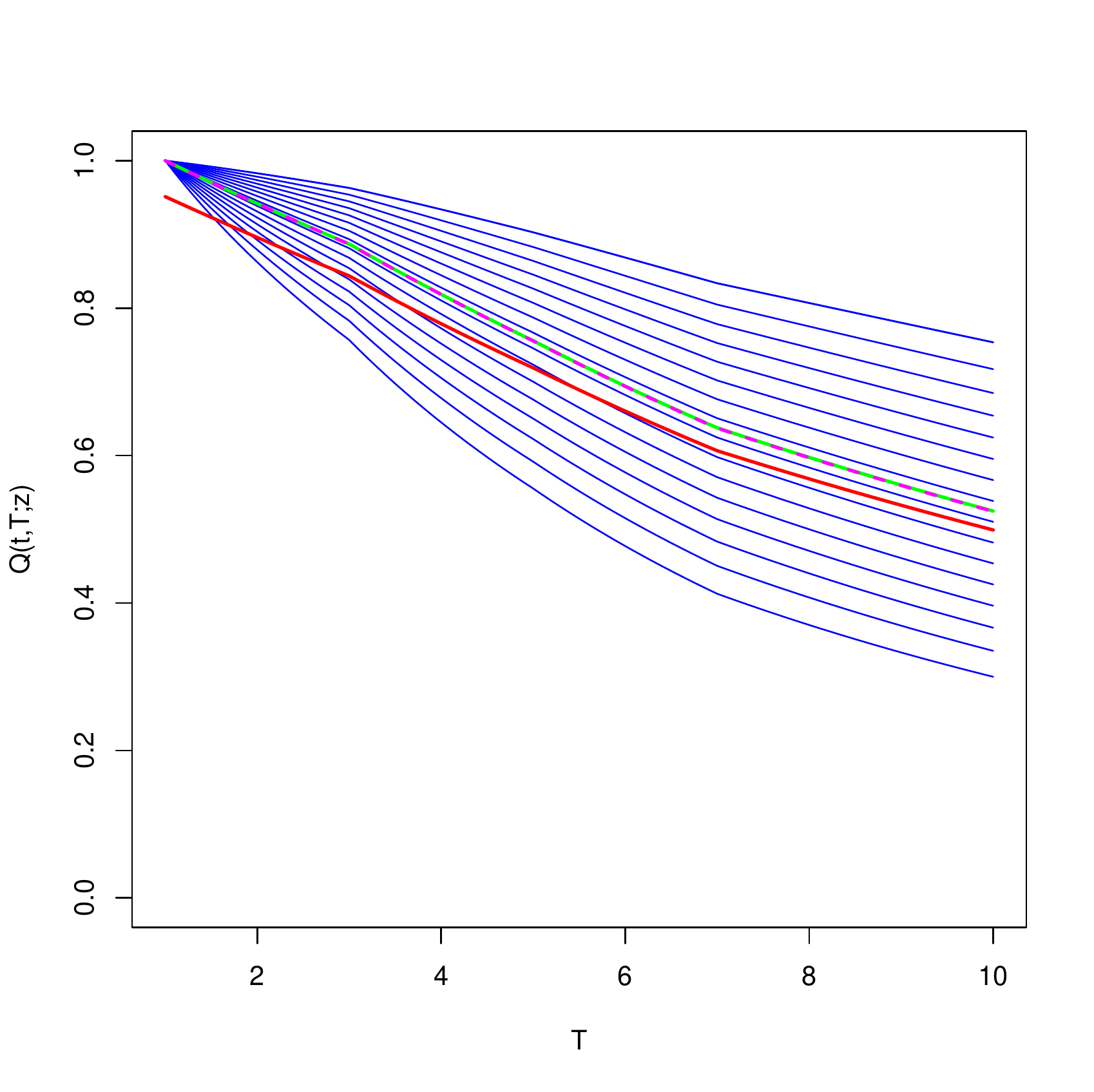}}
\hspace{0.5cm}
\subfigure[$\eta=0.25$]{\includegraphics[width=0.45\columnwidth]{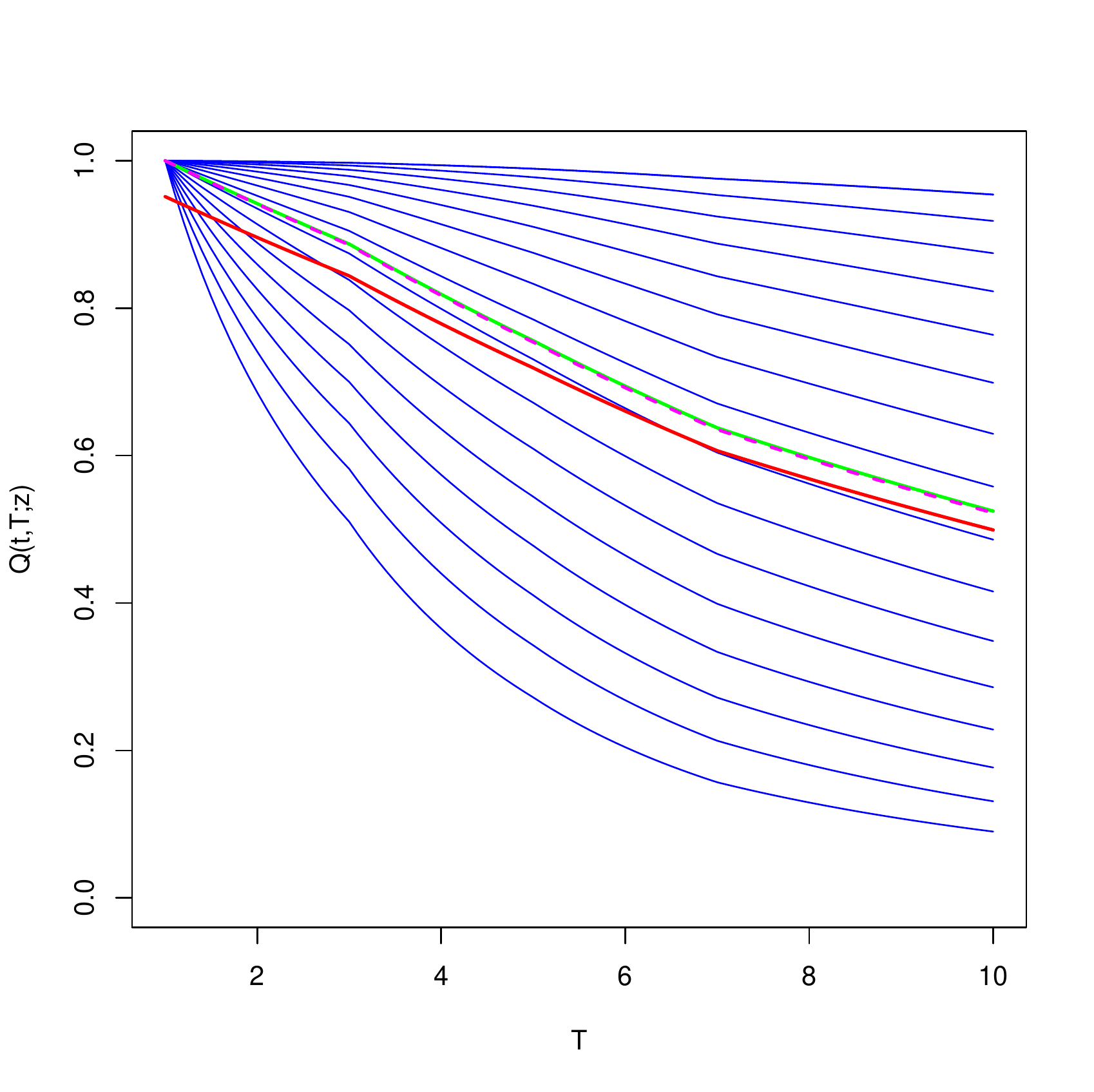}}
\caption{$S_0(T)$ (red), $S_0(T)/S_0(t)$ (green), $Q(t,T;z_i)$ (blue) and $\sum_{i=1}^{16} \omega_i Q(t,T;z_i)$ (magenta).}\label{fig:QtT}
\end{figure}
The boundary conditions for the cumulative distribution function $F_{Q_{t}(T)}(x)$ of $Q_{t}(T)$ are $F_{Q_{t}(T)}(0)=0$ and $F_{Q_{t}(T)}(1)=1$.

For $x\in(0,1)$, it is proven to be (see Appendix \ref{app:distribution})
\beqn F_{Q_{t}(T)}(x)&=&\Phi(z^\star) \nonumber\eeqn
where $z^\star=z^\star(t,T,x,\eta)$ is the (unique) root of the function
\beq G(z):=
x\Phi\left(m(t,t)+\sqrt{v(t)}z\right)-\Phi\left(m(t,T)+\sqrt{v(t)}z\right)\nonumber
\eeq
%
\subsection{Bivariate survival probability in the Gaussian copula setup}
%
Let us define the time-$t$ joint survival probability that $\tau_1>T_1$ and $\tau_2>T_2$ as
\beq
G_t(\vec{T})=\Q\{\tau_1>T_1,\tau_2>T_2|\filF_t\}\label{eq:JSP}
\eeq
We adopt a copula framework, where $G_t(T_1,T_2)$ depends on $t$
through the marginal (stochastic and correlated) distributions
$S^{1}_{t,T_1}$, $S^{2}_{t,T_2}$ and time-dependent set of
parameters $\Theta_t$ which is assumed to have finite variation\footnote{Note the difference between
$S_i(t)=S^i_{0}(t)=\Q\{\tau_i>t|\filF_0\}$ and
$S^i_t=S^i_{t}(t)=\Q\{\tau_i>t|\filF_t\}$}. In particular, the copula
if fixed, but its parameter (e.g. correlation) can be
time-dependent:   \beq
G_t(\vec{T})=C(S^1_{t}(T_1),S^2_{t}(T_2),\Theta_t) \nonumber\eeq

Observe that $G_t$ meets all the properties of multivariate cumulative distribution functions;
this is guaranteed by the fact that we map valid margins through a
copula $C(u,v,\Theta)$. However, it is clear from eq.~(\ref{eq:JSP})
that $G$ is a martingale, hence the SDE of $G$ must have no
\textit{dt} term. This imposes some restrictions on the dynamics
of the (meta) parameter $\Theta_t$.

It\^o's lemma yields
\beqn
dG_t(\vec{T})&=&\frac{\partial C}{\partial u}dS^1_{t}(T_1)+\frac{\partial C}{\partial v}dS^2_{t}(T_2)\nonumber\\
&&+\frac{1}{2}\left(\frac{\partial^2 C}{\partial u^2} d\langle S^1_{\cdot}(T_1),S^1_{\cdot}(T_1)\rangle_t+
\frac{\partial^2 C}{\partial v^2} d\langle S^2_{\cdot}(T_2),S^2_{\cdot}(T_2)\rangle_t\right)\nonumber\\
&&+\frac{\partial^2 C}{\partial u\partial v}d\langle
S^1_{\cdot}(T_1),S^2_{\cdot}(T_2)\rangle_t+\frac{\partial
C}{\partial \Theta} d\Theta_t \nonumber\eeqn
We now consider the bivariate Gaussian case, where correlated $\Phi$-martingales are plugged in a Gaussian copula:
\beq C(u,v,\Theta)=\Phi_2\left(\Phi^{-1}(u),\Phi^{-1}(v);r\right)\nonumber
\eeq
with
\beqn
S^i_{t}(T_i)&=&\Phi(X^i_{t}(T_i))\nonumber\\
dX^i_{t}(T_i) &=& \mu_i X^i_{t}(T_i)dt + \eta_i dW_t^i\nonumber\\
d\langle
W^1_{\cdot},W^2_{\cdot}\rangle_t&=&\rho dt\nonumber
\eeqn
Recall that $S^i_{t}(T_i)$ both need to be martingales, so that $\mu_i=\eta_i^2/2$. Allowing the Gaussian copula (correlation) parameter to be a deterministic function of time $r(t)$,
\beq
G_t(\vec{T})=\Phi_2\left(\Phi^{-1}(S^1_{t}(T_1)),\Phi^{-1}(S^2_{t}(T_2));r(t)\right)=\Phi_2\left(X^1_{t}(T_1),X^2_{t}(T_2);r(t)\right)\nonumber
\eeq
In this specific case, we get
\beqn
dG_t(\vec{T})&=&\frac{\partial \Phi_2}{\partial x}dX^1_{t}(T_1)+\frac{\partial \Phi_2}{\partial y}dX^2_{t}(T_2)\nonumber\\
&&+\frac{1}{2}\left(\frac{\partial^2 \Phi_2}{\partial x^2}d\langle X^1_{\cdot}(T_1),X^1_{\cdot}(T_1)\rangle_t
+\frac{\partial^2 \Phi_2}{\partial y^2}d\langle X^2_{\cdot}(T_2),X^2_{\cdot}(T_2)\rangle_t\right)\nonumber\\
&&+\frac{\partial^2 C}{\partial x\partial y}d\langle
X^1_{\cdot}(T_1),X^2_{\cdot}(T_2)\rangle_t+\frac{\partial
\Phi_2}{\partial r} dr(t) \nonumber\eeqn
The following derivatives are useful:
\beqn
\frac{d \phi(x)}{dx}&=&-x\phi(x)\nonumber\\
\frac{\partial \Phi_2}{\partial r}&=&\frac{\partial^2 \Phi_2}{\partial x\partial y}\nonumber\\
&=&\frac{1}{\sqrt{1-r^2}}\phi(x)\phi\left(\frac{y-rx}{\sqrt{1-r^2}}\right)\nonumber\\
&=&g(x,y,r)\nonumber\\
&=&g(y,x,r)\nonumber\\
\frac{\partial \Phi_2}{\partial x}&=&\phi(x)\Phi\left(\frac{y-rx}{\sqrt{1-r^2}}\right)\nonumber\\
&=&h(x,y,r)\nonumber\\
\frac{\partial^2 \Phi_2}{\partial x^2}&=&-x
h(x,y,r)-rg(x,y,r)\nonumber \eeqn
Noting that $d\langle
X^i_{\cdot}(T_i),X^j_{\cdot}(T_j)\rangle_t=\rho_{ij}\eta_i\eta_jdt$
with $\rho_{ii}=1$ and $\rho_{ij}=\rho$, and $r'(t)$ the
derivative of $r(t)$, the dynamics of $G_t(\vec{T})$ become
\beqn
dG_t(\vec{T})&=&\eta_1h\left(X^1_{t}(T_1),X^2_{t}(T_2),r(t)\right)dW^1_t+\eta_2h\left(X^2_{t}(T_2),X^1_{t}(T_1),r(t)\right)dW^2_t\nonumber\\
&&+\left(r'(t)+\rho\eta_1\eta_2-r(t)\frac{\eta_1^2+\eta_2^2}{2}\right)g\left(X^1_{t}(T_1),X^2_{t}(T_2),r(t)\right)dt\nonumber
\eeqn
Martingality is guaranteed provided that the \textit{dt} term is zero, that is
\beq
r(t)=\frac{2\rho\eta_1\eta_2}{\eta_1^2+\eta_2^2}+k\e^{\frac{\eta_1^2+\eta_2^2}{2}t}
\eeq
Because the first term is a constant but the second grows without bound, the $r(t)\in[-1,1]$ condition imposes $k=0$ if $\eta_1$ and $\eta_2$ are not both null. The only valid case is thus to set the Gaussian copula correlation parameter $r$ to the specific value $\frac{2\rho\eta_1\eta_2}{\eta_1^2+\eta_2^2}$, depending on the variance and correlation of the latent processes underlying the marginal Az\'ema supermartingales.

Finally, the bivariate Az\'ema supermartingale has dynamics
\beqn
dG_t(\vec{t})&=&dG_t(\vec{T})|_{\vec{T}=(t,t)}+\xi^1_tdS_0^1(t)+\xi^2_tdS_0^2(t)\\
\xi^1_t&=&\frac{\e^{\mu_1 t}h\Big(\Phi^{-1}(S^1_{t}),\Phi^{-1}(S^2_{t}),r(t)\Big)}{\phi\Big(\Phi^{-1}(S_0^1(t))\Big)}\\
\xi^2_t&=&\frac{\e^{\mu_2 t}h\Big(\Phi^{-1}(S^2_{t}),\Phi^{-1}(S^1_{t}),r(t)\Big)}{\phi\Big(\Phi^{-1}(S_0^2(t))\Big)}
\eeqn
where, $\xi^i_t$ are the multivariate equivalent to $\zeta_t$ and ensure that
\beq
\E[G_t(t,t)]=G_0(t,t)=\Phi_2\Big(\Phi^{-1}(S_0^1(t)),\Phi^{-1}(S_0^2(t)),r(0)\Big)
\eeq
We conclude this section with Fig.~\ref{fig:BivStT} which illustrates the joint Az\'ema supermartingale processes $G_t(t,t)=\E[\tau_1>t,\tau_2>t|\filF_t]$ and the joint survival probability martingale $G_t(T,T)=\E[(\tau_1\wedge\tau_2)>T|\filF_t]$ for different correlation levels (using the same pairs of paths of Brownian motions).
\begin{figure}
\centering
\includegraphics[width=0.98\columnwidth]{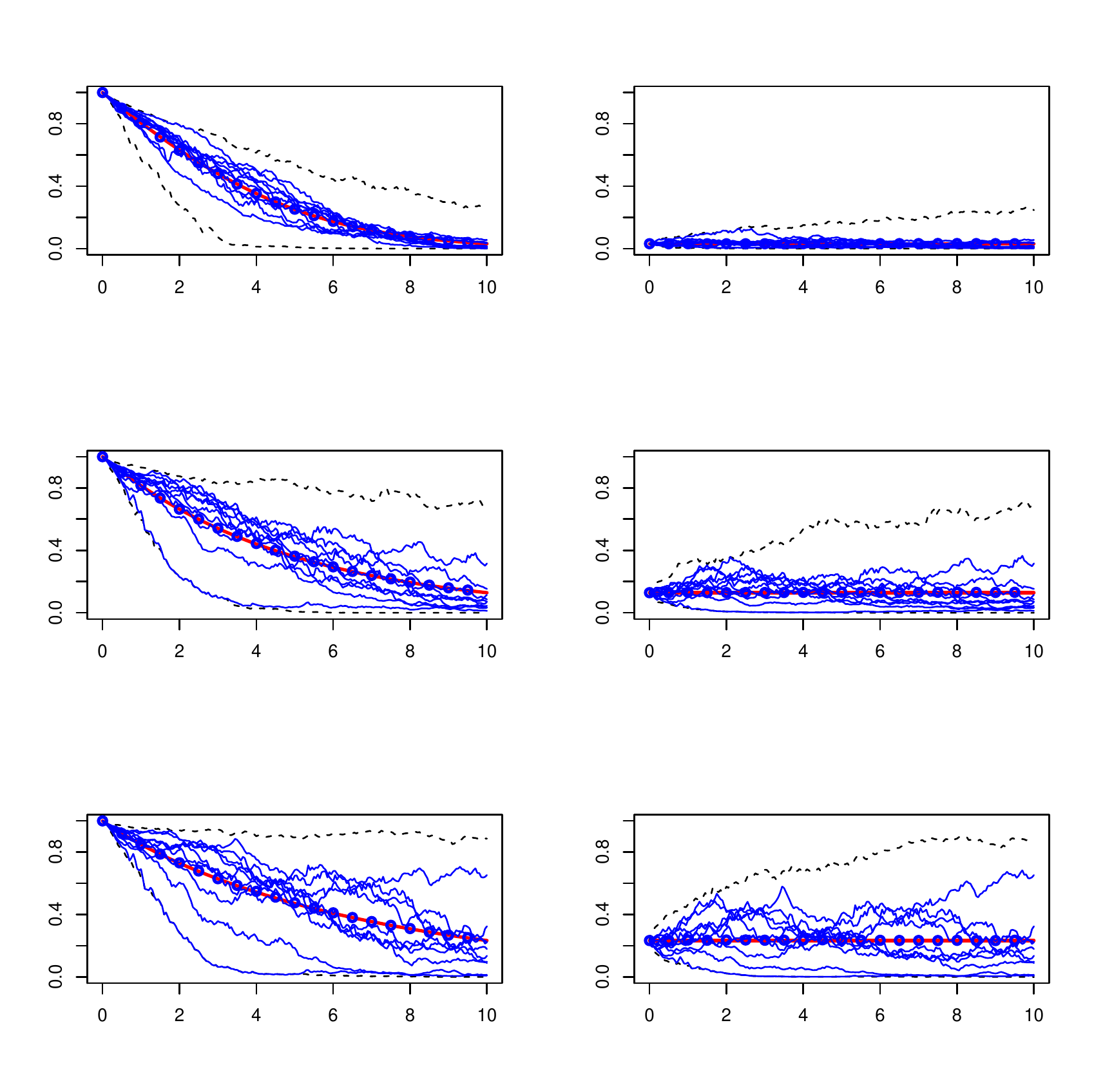}
\caption{Bivariate Az\'ema supermartingale processes $G_t(t,t)=\E[\tau_1>t,\tau_2>t|\filF_t]$ (left) and bivariate survival probability martingale $G_t(T,T)=\E[(\tau_1\wedge\tau_2)>T|\filF_t]$ (right) with Brownian correlations $\rho=\{-80\%,0,+80\%\}$ (top down). Ten sample paths are shown. There is no simulation error as the exact solution is known (up to numerical errors in the evaluation of the bivariate cumulative Normal distribution). We used constant hazard rates for $S_0^i(t)$ ($h_1=8\%$ and $h_2=12.5\%$), constant volatilities ($\eta^1=15\%$ and $\eta^2=25\%$) time step of $0.05$. Blue dots on right panels show sample average based on $1,000$ paths. Black dashed lines show the min and max envelopes based on the sample set. }\label{fig:BivStT}
\end{figure}
%
\section{Conclusion and future work}
%
In a first part of the paper, the conditions for a local martingale to be a genuine martingale have been reviewed and specialized to bounded processes. We have introduced the concept of \textit{conic processes} as stochastic processes which range is finite and non-decreasing with respect to time. We have shown that martingales being locally bounded are in fact conic martingales. It has been explained how a martingale evolving between two constant bounds with given separable diffusion coefficient $\sigma(t,x)=g(t)h(x)$ can be obtained by mapping a stochastic process $X$ which diffusion coefficient is $g(t)$ through the function $F$ solving a first-order autonomous non-linear ODE featuring $h(x)$. This results is interesting for simulation purposes as the paths of $F(X)$ will then stay within the correct range.

The case of martingales evolving in the standard interval $[0,1]$ received specific attention. Several examples have been provided for which existence and uniqueness results have been established. The mapping $F$ consisting of the standard Normal cumulative distribution $\Phi$ proves to be interesting. It allows to turn any It\^o integral into a martingale bounded in $[0,1]$. Moreover, the $\Phi$-martingale built by mapping a Vasicek process through $\Phi$ proves to be a tractable $[0,1]$-martingale that does not attain the bounds in finite time. Its statistics have been computed analytically, and its distribution is proven to converge to a Bernoulli with parameter given by the initial value of the process. It has been shown that it is the only $[0,1]$-martingale that can be obtained by mapping a specific class of Gaussian processes (called \textit{Gaussian diffusions}) in a time-homogeneous way. The $\Phi$-martingale completes the class of the possible martingales that can be obtained by such means. The other processes are the constant, the rescaled Brownian motion and the driftless Geometric Brownian motion. Interestingly, each of these martingales correspond to a specific range.

Martingales in $[0,1]$ have been applied to the construction of Az\'ema supermartingales out of the Cox setup. They obviously meet the range constraint and benefits from automatic calibration. To our knowledge this is the first analytically tractable approach satisfying this requirement for all valid initial default probability curves. Similarly, one can built a set of conditional survival probability curves evolving in time with respect to a risk factor modeled as a Brownian motion. This was extended to the modeling of mutlivariate stochastic survival probabilities.

This work suggests several routes for future research. For instance, all the conic martingales derived in this paper have constant cones. Naive extension of the above construction schemes lead to SDEs that do not meet the usual existence criteria. It is not clear yet whether bounded martingales with time-dependent cones can be found explicitly. Another route for future research deals with non-continuous bounded martingales. Finally, we believe this work opens the door for alternative approaches for the risk management of products depending on default probability curves, like for example the modeling of exposure profiles of credit-linked financial instruments.

\section{Appendix}

\subsection{Derivation of the law of $X_t$ in the Vehulst martingale $Y_t=\exp\{-\lambda X_t\}$}
\label{app:verhulst}
%
Let us note the geometric Brownian motion $\Theta_t=g(t,W_t)$ where $g(t,x):=X_0\xi(t,x)$, $\xi(t,x):=\exp\left\{\nu x-\frac{\nu^2 t}{2}\right\}$ and let $\hat{\Theta}_t:=\int_0^t \Theta_s ds$. We seek for the density of $X_t=\Theta_t/\left(1-\mu\hat{\Theta}_t\right)$. To that end, we are interested in the joint density of $\left(\Theta_t,\hat{\Theta}_t\right)$. From the conditional density $p_{\hat{\Theta}_t|W_t}(y,x)$ of $\hat{\Theta}_t$ conditional upon the terminal value of the Brownian motion $W_t=x$, one gets
\beqn
 f_{X_t}(z)&=&\int_{-\infty}^\infty p_{\hat{\Theta}_t|W_t}\left(\frac{z-g(t,x)}{\mu z}\Big|x\right)\frac{\phi(x/\sqrt{t})}{\sqrt{t}}dx
\nonumber\eeqn
This expression features the density of the integral $\hat{\Theta}_t$ of a geometric Brownian motion conditional upon the terminal value of the Brownian motion $W_t$. This expression is quite important in finance, and appears in Asian options. Therefore, it received some attention and Marc Yor derived the corresponding expression by using relationships with Bessel processes~\cite{Yor92}.

Following Yor's notations, define
\beqn
A_{t}(\nu)&=&\int_0^t\exp\left\{2(W_s+\nu s)\right\}ds \nonumber
\eeqn
Then, observed that from the scaling property of Brownian motion,
\beqn
\int_0^t \exp\left\{aW_s+bs\right\}ds&\sim&\frac{4}{a^2}A_{a^2t/4}(2b/a^2)\nonumber
\eeqn
Setting $a=\eta$, $b=-\eta^2/2$, $\nu=2b/a^2=-1$ and $t'=\eta^2t/4$, we obtain
\beq
\left(\hat{\Theta}_t,W_t\right)\sim\left(\frac{4X_0}{\eta^2}A_{t'}(-1),\frac{2}{\eta}W_{t'}\right)
\nonumber\eeq
On the other hand, if $\textbf{Y}=\textbf{A}\textbf{X}$ where $\textbf{A}$ is an invertible matrix and $\textbf{X},\textbf{Y}$ random (column) vectors, then the density of $\textbf{Y}$ is given by $f_{\textbf{Y}}(\textbf{y})=\frac{1}{|\det \textbf{A}|}f_{\textbf{X}}(\textbf{A}^{-1}\textbf{y})$. With $\textbf{A}=Diag(\frac{4X_0}{\eta^2},\frac{2}{\eta})$,
\beqn
 f_{X_t}(z)&=&\frac{\eta^3}{8X_0\sqrt{t'}}\int_{-\infty}^\infty p_{A_{t'}(-1)|W_{t'}}\left(\frac{\eta^2}{4 X_0}\left(\frac{z-g\left(t,\frac{\eta x}{2}\right)}{\mu z}\right)\Big|\frac{\eta x}{2}\right)\phi\left(\frac{\eta x}{2\sqrt{t'}}\right)dx\nonumber\\
&=&\frac{\eta ^2}{4X_0\sqrt{t'}}\int_{-\infty}^\infty p_{A_{t'}(-1)|W_{t'}}\left(\frac{\eta^2}{4 X_0}\left(\frac{z-g\left(t,y\right)}{\mu z}\right)\Big|y\right)\phi\left(y/\sqrt{t'}\right)dy
\nonumber\eeqn
Yor argued that it is enough to study the conditional law of $A_t(0)$ since $p_{A_{t}(\nu)|W_{t}}\left(z|y+\nu_t\right)=p_{A_{t}(0)|W_{t}}\left(z|y\right)$. To see this, let us define the measure $\tilde{\Q}$ according to the Radon-Nikodym derivative process $\frac{d\tilde{\Q}(\omega)}{d\Q(\omega)}\Big|_{\filF_t}=\xi(t,W_t)$. From Girsanov's theorem, $\tilde{W}_t=W_t+\nu t$ is a $\tilde{\Q}$-Brownian motion, and the claim follows:
\beqn
p_{A_{t}(\nu)|W_{t}}\left(z|y\right)&=&\frac{\int_\Omega \ind_{\{A_t(\nu)=z,W_t=y\}}d\Q(\omega)}{\int_\Omega \ind_{\{W_t=y\}}d\Q(\omega)}\nonumber\\
&=&\frac{\int_\Omega \ind_{\{\int_0^t\exp\left\{2\tilde{W}_s\right\}ds=z,\tilde{W}_t=y+\nu t\}}d\Q(\omega)}{\int_\Omega \ind_{\{\tilde{W}_t=y+\nu t\}}d\Q(\omega)}\nonumber\\
&=&\frac{\int_\Omega \ind_{\{\int_0^t\exp\left\{2\tilde{W}_s\right\}ds=z,\tilde{W}_t=y+\nu t\}}\xi(t,y)d\tilde{\Q}(\omega)}{\int_\Omega \ind_{\{\tilde{W}_t=y+\nu t\}}\xi(t,y)d\tilde{\Q}(\omega)}\nonumber\\
&=&\frac{\xi(t,y)\int_\Omega \ind_{\{A_t(0)=z,W_t=y+\nu t\}}d\Q(\omega)}{\xi(t,y)\int_\Omega \ind_{\{W_t=y+\nu t\}}d\Q(\omega)}\nonumber\\
&=&p_{A_{t}(0)|W_{t}}\left(z|y+\nu t\right)\nonumber
\eeqn
Finally, the density $f_{X_t}$ is obtained from the law of $A_t(0)$ conditioned upon the terminal value of the Brownian motion $W_t$, $a_t(y,z)=p_{A_{t}(0)|W_{t}}\left(z|y\right)$, which is proven in~\cite{Yor92} to be:
\beqn
a_t(y,z)&=&\frac{\sqrt{t}}{z\phi(y/\sqrt{t})}\exp\left(-\frac{1+e^{2y}}{2z}\right)\theta_{e^{y}/z}(t)\nonumber\\
\theta_r(u)&=&\frac{1}{\sqrt{2u\pi ^3}}\exp\left(\frac{\pi^2}{2u}\right)\Psi_r(u)\nonumber\\
\Psi_r(u)&=&\int_{0}^\infty \exp\left(\frac{-y^2}{2u}\right)\exp\left(-r\cosh y\right)\sinh(y)\sin\left(\frac{\pi y}{u}\right)dy\nonumber
\eeqn
and we obtain
\beqn
f_{X_t}(z)&=&\frac{\eta^2}{4X_0\sqrt{t'}}\int_{-\infty}^\infty a_{t'}\left(y- t',\frac{\eta^2}{4 X_0}\left(\frac{z-g\left(t',y- t'\right)}{\mu z}\right)\right)\phi\left(y/\sqrt{t'}\right)dy\nonumber
\eeqn
%
\subsection{Collapsing property of bounded martingales}
\label{app:collapse}
%
It has been proven that when $X $ is a Gaussian diffusion  with diffusion coefficient $\eta$ and drift $(\eta^2/2)x$ then $Y=\Phi(X)$ is a martingale bounded in $[0,1]$ which converges in distribution to a $Bernoulli(Y_0)$. This proof was easy as the distribution of $Y_t$ is known analytically. However, it is likely that autonomous martingales of the form $F(X)$ where $X$ is a free process and the image of $F$ is a compact interval $[a,b]$ will share the same ``collapsing'' feature. Although we do not give a formal proof, we provide an intuitive development below. We further discuss which form of mappings $F$ could potentially \textit{not} have this feature.

Let $F$ (assumed to be strictly increasing and $\mathcal{C}^2$) and $f$ be unimodal (i.e. $f'(x)$ is first positive, then vanishes at some point $x^\star$ and then remains negative) when $\eta(t,X_t)=\eta$. Recall the SDE followed by  $X$:
\beq
dX_t=-\frac{\eta^2}{2}\frac{f'(X_t)}{f(X_t)}dt+\eta dW(t)=\frac{\eta^2}{2}\psi(X_t)dt+\eta dW(t)
\eeq
Because $f'(x)<0$ for all $x>x^\star$ and $f'(x)>0$ for all $x>x^\star$, we can see from the above SDE that $X $ has a positive drift when being above $x^\star$ and a negative drift otherwise. As long as we choose $F$ such that $f$ is unimodal, then the process will tend to diverge, and same will hold true for $Y =F(X )$. In other words, conic martingales $Y $ obtained via an underlying process $Y=F(X)$ will be attracted towards one of the boundaries when the bounds $(a,b)$ are constant and $f=F'$ is unimodal.

Although it could take quite some time before $Y $ collapses to $a$ or $b$, it is very likely to happen, unless $F$ is chosen to have some specific properties, and the above development allows us to understand which properties may break this attraction.

A first possibility of course is to use a vanishing time-dependent diffusion coefficient $\eta(t)$. If $\eta(t)$ collapses to zero, the   process $X$ will be frozen, and so are the paths of the $Y_t$ process. However, we argue that this is not the only way to prevent all paths to converge to one of the bounds. Although we do not give formal proof, we claim that this can be achieved by choosing $F$ so that $F'=f$ is bimodal (i.e. $f'$ changes sign between the two modes) and $X_0$ belongs to the interval defined by the two modes. Consider for example $F(x)=1/2\left(\Phi\left(\frac{x-(X_0+\mu)}{s}\right)+\Phi\left(\frac{x-(X_0-\mu)}{s}\right)\right)$ for $s>0$ and $\mu$ large enough to ensure that the sign of $f'$ changes between the two modes. This function maps $\mathbb{R}$ to the unit interval $[0,1]$. Then, the fact that $X_t$ falls in the ``dip'' $[X_0-\mu,X_0+\mu]$ will create a pulling effect such that $X_t$ will tend to stay within this interval.\footnote{Observe that this pulling effect does not impact the martingale property of $Y$, since this is compensated by the function $F$, just like the fact that $X$ is a diverging process when $F=\Phi$ is used did not impact the martingality of $Y$ in the unimodal case.} The effect of the bimodal nature of $f$ is to partly prevent all paths to collapse to one of the boundaries (the center of the distribution would not be empty anymore; the probability to be in arbitrarily small neighborhood of the bounds would be non-zero, but would not sum to 1 either). It remains to formally prove that $\Q\{X_t\in [X_0-\mu,X_0+\mu]\}>0$ as $t\to\infty$ for any $\mu>0$. This is cumbersome since we do not have an analytical expression for the distribution of $X_t$. However, Monte Carlo simulations or PDE solver seem to confirm that this is effectively the case: only part of the paths collapse to the boundary. The ``sharpness of the dip'' (which can be tuned by playing with $s$) determines the probability that $Y_t$ lies in $[F(X_0-\mu),F(X_0+\mu)]$ as $t\to\infty$. However, the process $X_t$ built according to the above procedure has a stable stationary point (zero drift, or equivalently $f'(x)=0$) at $X_0$ (zero drift, and in the neighborhood, the drift tends to pull $X_t$ back to $X_0$), and unstable stationary points at $X_0\pm\mu$ (zero drift, but when $X_t$ moves around these points, the effect of the drift is to push $X_t$ away from those). It is then likely that the paths, instead of collapsing (asymptotically) to either 0 or 1 almost surely, now asymptotically collapse to $\{0,1\}$ with some probability $p$, but $X_t$ has a non-zero probability $1-p$ to be in the interval $[X_0-\mu,X_0+\mu]$ even in the limit $t\to\infty$. In particular, we expect to have $\lim_{t\to\infty}\Q\{Y_t\in(0,F(X_0-\mu)]\}=\lim_{t\to\infty}\Q\{[Y_t\in[F(X_0+\mu),1)\}=0$, while for any $\epsilon>0$, $\lim_{t\to\infty}\Q\{Y_t\in(1-\epsilon,1]\}=\lim_{t\to\infty}\Q\{Y_t\in[0,\epsilon)\}>0$ and $\lim_{t\to\infty}\Q\{Y_t\in[F(X_0-\mu),F(X_0+\mu)]\}>0$.

\subsection{Distribution of $Q_{t}(T)$} \label{app:distribution}
%
In this section we show that for $x\in(0,1)$, $\eta^2<\infty$ and $S_0(T)<S_0(t)$, the cumulative distribution function $F_{Q_{t}(T)}(x)$ of $Q_{t}(T)$ is given by $\Phi(z^\star(t,T,x,\eta))$.

First, we notice that $\eta Z_t\sim \sqrt{v(t)}Z$ where $Z$ a standard Normal variable, the distribution function of $Q_{t}(T)$ is given by
\beqn
\Q\{Q_{t}(T)\leq x\}&=&\Q\left\{ Q(t,T;Z)\leq x\right\}\\
&=&\Q\left\{ \Phi\left(m(t,T)+\sqrt{v(t)}Z\right)\leq x\Phi\left(m(t,t)+\sqrt{v(t)}Z\right)\right\}\\
&=&\Q\left\{ \Phi\left(y\right)\leq x\Phi\left(m+y\right)\right\}\\
y&:=&m(t,T)+\sqrt{v(t)}Z\\
m&:=&m(t,t)-m(t,T)
\eeqn
From the above notations, we can write
\beq
F_{Q_{t}(T)}(x)=\int_{z:G(z)>0}d\Phi(z)
\eeq
Moreover, by definition of $z^\star$, $G(z^\star)=0$. It remains to show that $\{z:G(z)>0\}=(-\infty,z^\star]$, in which case we have the claim
\beq
\int_{z:G(z)>0}d\Phi(z)=\Phi(z^\star)
\eeq
Clearly, under our assumptions on $S_0(T)$, $m>0$. We define
\beq
\tilde{G}(y)=x\Phi(m+y)-\Phi(y)
\eeq
so that $\tilde{G}(y)=\tilde{G}(m(t,T)+\sqrt{v(t)}z)=G(z)$. Our purposes is to show that for $m>0$ and $x\in(0,1)$, $\{y:\tilde{G}(y)>0\}=(-\infty,y^\star]$ with $z^\star=(y^\star-m(t,T))/\sqrt{v(t)}$, since the claim then results from a continuity argument.

First, we notice that $\tilde{g}(y)=d\tilde{G}(y)/dy$ has one single root, $y_0=\ln(x)/m-m/2$, and $\tilde{g}(y)>0$ for $y<y_0$ whilst $\tilde{g}(y)<0$ for $y>y_0$. Because $\tilde{G}(y)=\int_{-\infty}^y \tilde{g}(u)du$ (the integration constant is zero as $\lim_{y\downarrow -\infty}\tilde{G}(y)=0$), the smallest root of $\tilde{G}$, $y_1^\star$ is larger than $y_0$; $y_1^\star>y_0$. However, for $y>y_0$, $\tilde{g}(y)<0$, meaning that on the right of its first root, $\tilde{G}(y)$ is strictly decreasing from $0$. Function $\tilde{G}(y)<0$ for $y>y^\star_1$, showing that if it exists, $y_1^\star$ is the unique root $y^\star$ of the continuous function $\tilde{G}$ or equivalently, that $G$ admits a unique root $z^\star=(y^\star-m(t,T))/\sqrt{v(t)}$.


\ifdefined \MyBib
    \bibliography{\MyBib}
  \bibliographystyle{unsrtnat}
\fi
\end{document}